\documentclass[a4paper,12pt]{article}
\usepackage[top=2.5cm,bottom=2.5cm,left=2.5cm,right=2.5cm]{geometry}
\usepackage{cite, amsmath, amssymb}
\usepackage[margin=1cm,%
font=small,%
format=hang,%
labelsep=period,%
labelfont=bf]{caption}
\usepackage{amssymb,amsfonts,amsmath,latexsym,epsf,tikz,url}
\usepackage[usenames,dvipsnames]{pstricks}
\usepackage{pstricks-add}
\usepackage{epsfig}
\usepackage{pst-grad} 
\usepackage{pst-plot} 
\usepackage[space]{grffile} 
\usepackage{etoolbox} 
\makeatletter 
\patchcmd\Gread@eps{\@inputcheck#1 }{\@inputcheck"#1"\relax}{}{}
\makeatother

\newtheorem{theorem}{Theorem}[section]
\newtheorem{proposition}[theorem]{Proposition}

\newtheorem{corollary}[theorem]{Corollary}

\newcommand{\proof}{\noindent{\bf Proof.\ }}
\newcommand{\qed}{\hfill $\square$\medskip}

\begin{document}

\title{Sombor Index of Polymers}
\author{
Saeid Alikhani$^{1,}$\footnote{Corresponding author}
\and
Nima Ghanbari$^2$
}

\date{\today}

\maketitle

\begin{center}
$^1$Department of Mathematics, Yazd University, 89195-741, Yazd, Iran\\
$^2$Department of Informatics, University of Bergen, P.O. Box 7803, 5020 Bergen, Norway\\
{\tt  alikhani@yazd.ac.ir, Nima.ghanbari@uib.no}
\end{center}

\begin{center}
(Received February 28, 2021) 
\end{center}

\begin{abstract}
  Let $G=(V,E)$ be a finite simple graph. The Sombor  index $SO(G)$ of $G$
  is defined as $\sum_{uv\in E(G)}\sqrt{d_u^2+d_v^2}$, where $d_u$ is the degree of
  vertex $u$ in $G$. Let $G$ be a connected graph constructed from pairwise disjoint connected graphs $G_1,\ldots ,G_k$ by selecting a vertex of $G_1$, a vertex of $G_2$, and identifying these two
  vertices. Then continue in this manner inductively. We say that $G$ is a polymer graph, obtained by point-attaching   from monomer units $G_1,...,G_k$.  
  In this paper, we consider some  particular cases  of these graphs that  are  of importance in chemistry  and study their Sombor index. 
\end{abstract}

\baselineskip=0.30in

\section{Introduction}
 A molecular graph is a simple graph such that its vertices correspond to the atoms and the edges to the bonds of a molecule. 
Let $G = (V, E)$ be a finite, connected, simple graph. We denote the degree of a vertex $v$ in $G$ by $d_v$. 
A topological index of $G$ is a real number related to $G$. It does not depend on the labeling or pictorial representation of a graph. The Wiener index $W(G)$ is the first distance based topological index defined as $W(G) = \sum_{\{u,v\}\subseteq G}d(u,v)=\frac{1}{2} \sum_{u,v\in V(G)} d(u,v)$ with the summation runs over all pairs of vertices of $G$ \cite{20}.
The topological indices and graph invariants based on distances between vertices of a graph are widely used for characterizing molecular graphs, establishing relationships between structure and properties of molecules, predicting biological activity of chemical compounds, and making their chemical applications.  The
Wiener index is one of the most used topological indices with high correlation with many physical and chemical indices of molecular compounds \cite{20}. 
Recently in \cite{Gutman2}  a new vertex-degree-based molecular structure descriptor was put forward, the Sombor index, defined as  
$$SO(G) =\sum_{uv\in E(G)}\sqrt{d_u^2+d_v^2}.$$
Cruz, Gutman and Rada in \cite{AMC} characterized the graphs extremal  with respect to this index over the chemical graphs, chemical trees and hexagon systems (see \cite{Sombor}).
In \cite{Deng}, the chemical importance of the Sombor index has investigated and it is shown that this index is useful in predicting physico‐chemical properties with high accuracy compared to some well‐established and often used indices. Also a sharp upper bound for the Sombor index among all molecular trees with fixed numbers of vertices has obtained, and  those molecular trees achieving the extremal value has characterized. 
In \cite{Red} the predictive and discriminative potentials of Sombor index, reduced Sombor index, and average Sombor index  examined. All three topological molecular descriptors showed good predictive potential. 
In \cite{Symmetry} some novel lower and upper bounds on the Sombor index of graphs has presented  by using some graph
parameters, especially, maximum and minimum degree. Moreover,  several relations on Sombor index with the first and second Zagreb
indices of graphs obtained.  The mathematical relations between the Sombor index and some other well-known degree-based descriptors investigated in \cite{Wang}.

\medskip

In this paper, we consider the Sombor index of polymer graphs. Such graphs can be decomposed into subgraphs that we call monomer units. Blocks of graphs are particular examples of monomer units, but a monomer unit may consist of several blocks. For convenience, the definition of these kind of graphs will be given in the next  section.  In Section 2,  the Sombor index of some graphs are computed  from their monomer units. In Section 3, we apply the  results of Section 2, in order to obtain the Sombor index  of
families of graphs that are of importance in chemistry.

\section{Sombor index of polymers}

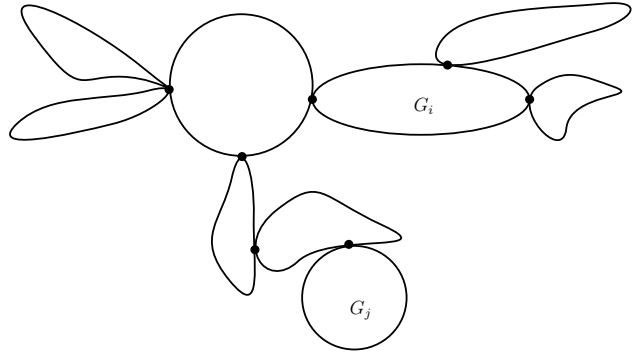
\begin{figure}
	\begin{center}
		\psscalebox{0.6 0.6}
{
\begin{pspicture}(0,-4.819607)(13.664668,2.90118)
\pscircle[linecolor=black, linewidth=0.04, dimen=outer](5.0985146,1.0603933){1.6}
\pscustom[linecolor=black, linewidth=0.04]
{
\newpath
\moveto(11.898515,0.66039336)
}
\pscustom[linecolor=black, linewidth=0.04]
{
\newpath
\moveto(11.898515,0.26039338)
}
\pscustom[linecolor=black, linewidth=0.04]
{
\newpath
\moveto(12.698514,0.66039336)
}
\pscustom[linecolor=black, linewidth=0.04]
{
\newpath
\moveto(10.298514,1.0603933)
}
\pscustom[linecolor=black, linewidth=0.04]
{
\newpath
\moveto(11.098515,-0.9396066)
}
\pscustom[linecolor=black, linewidth=0.04]
{
\newpath
\moveto(11.098515,-0.9396066)
}
\pscustom[linecolor=black, linewidth=0.04]
{
\newpath
\moveto(11.898515,0.66039336)
}
\pscustom[linecolor=black, linewidth=0.04]
{
\newpath
\moveto(11.898515,-0.9396066)
}
\pscustom[linecolor=black, linewidth=0.04]
{
\newpath
\moveto(11.898515,-0.9396066)
}
\pscustom[linecolor=black, linewidth=0.04]
{
\newpath
\moveto(12.698514,-0.9396066)
}
\pscustom[linecolor=black, linewidth=0.04]
{
\newpath
\moveto(12.698514,0.26039338)
}
\pscustom[linecolor=black, linewidth=0.04]
{
\newpath
\moveto(14.298514,0.66039336)
\closepath}
\psbezier[linecolor=black, linewidth=0.04](11.598515,1.0203934)(12.220886,1.467607)(12.593457,1.262929)(13.268515,1.0203933715820312)(13.943572,0.7778577)(12.308265,0.90039337)(12.224765,0.10039337)(12.141264,-0.69960666)(10.976142,0.5731798)(11.598515,1.0203934)
\psbezier[linecolor=black, linewidth=0.04](4.8362556,-3.2521083)(4.063277,-2.2959895)(4.6714916,-1.9655427)(4.891483,-0.99004078729821)(5.111474,-0.014538889)(5.3979383,-0.84551746)(5.373531,-1.8452196)(5.349124,-2.8449216)(5.6092343,-4.208227)(4.8362556,-3.2521083)
\psbezier[linecolor=black, linewidth=0.04](8.198514,-2.0396066)(6.8114076,-1.3924998)(6.844908,-0.93520766)(5.8785143,-1.6996066284179687)(4.9121203,-2.4640057)(5.6385145,-3.4996066)(6.3385143,-2.8396065)(7.0385146,-2.1796067)(9.585621,-2.6867135)(8.198514,-2.0396066)
\pscircle[linecolor=black, linewidth=0.04, dimen=outer](7.5785146,-3.6396067){1.18}
\psdots[linecolor=black, dotsize=0.2](11.418514,0.7403934)
\psdots[linecolor=black, dotsize=0.2](9.618514,1.5003934)
\psdots[linecolor=black, dotsize=0.2](6.6585145,0.7403934)
\psdots[linecolor=black, dotsize=0.2](3.5185144,0.96039337)
\psdots[linecolor=black, dotsize=0.2](5.1185145,-0.51960665)
\psdots[linecolor=black, dotsize=0.2](5.3985143,-2.5796065)
\psdots[linecolor=black, dotsize=0.2](7.458514,-2.4596066)
\rput[bl](8.878514,0.42039338){$G_i$}
\rput[bl](7.478514,-4.1196065){$G_j$}
\psbezier[linecolor=black, linewidth=0.04](0.1985144,0.22039337)(0.93261385,0.89943534)(2.1385605,0.6900083)(3.0785143,0.9403933715820313)(4.0184684,1.1907784)(3.248657,0.442929)(2.2785144,0.20039338)(1.3083719,-0.042142253)(-0.53558505,-0.45864862)(0.1985144,0.22039337)
\psbezier[linecolor=black, linewidth=0.04](2.885918,1.4892112)(1.7389486,2.4304078)(-0.48852357,3.5744174)(0.5524718,2.1502930326916756)(1.5934672,0.7261687)(1.5427756,1.2830372)(2.5062277,1.2429687)(3.46968,1.2029002)(4.0328875,0.5480146)(2.885918,1.4892112)
\psellipse[linecolor=black, linewidth=0.04, dimen=outer](9.038514,0.7403934)(2.4,0.8)
\psbezier[linecolor=black, linewidth=0.04](9.399693,1.883719)(9.770389,2.812473)(12.016343,2.7533927)(13.011008,2.856550531577144)(14.005673,2.9597082)(13.727474,2.4925284)(12.761896,2.2324166)(11.796317,1.9723049)(9.028996,0.9549648)(9.399693,1.883719)
\end{pspicture}
}
	\end{center}
	\caption{\label{Figure1} A polymer graph with monomer units  $G_1,\ldots , G_k$.}
\end{figure}

Let $G$ be a connected graph constructed from pairwise disjoint connected graphs
$G_1,\ldots ,G_k$ as follows. Select a vertex of $G_1$, a vertex of $G_2$, and identify these two vertices. Then continue in this manner inductively.  Note that the graph $G$ constructed in this way has a tree-like structure, the $G_i$'s being its building stones (see Figure \ref{Figure1}).  Usually  say that $G$ is a polymer graph, obtained by point-attaching from $G_1,\ldots , G_k$ and that $G_i$'s are the monomer units of $G$. A particular case of this construction is the decomposition of a connected graph into blocks (see \cite{Deutsch}). 
By the definition of the Sombor index, we have the following easy result: 

\begin{proposition}
 Let $G$ be a polymer graph with composed of 	monomers $\{G_i\}_{i=1}^k$. Then 
$$SO(G)> \sum_{i=1}^k SO(G_i).$$
\end{proposition} 
We consider some  particular cases  of these graphs  and study their Sombor   index. 
As an example of point-attaching graph,   consider the graph $K_m$ and $m$ copies of  $K_n$. By definition, the graph $Q(m, n)$ is obtained by identifying each vertex of $K_m$ with a vertex of a unique $K_n$. The graph $Q(5,4)$ is shown in Figure \ref{qmn}.

	\begin{theorem}
	For the graph $Q(m,n)$ (see Figure \ref{qmn}), and $n\geq 2$ we have:
\begin{align*}
SO(Q(m,n))&=m(\frac{(m+n-2)(m-1)}{2}+(n-1)^2(\frac{n}{2}-1))\sqrt{2}\\
	&\quad +m(n-1)\sqrt{(m+n-2)^2+(n-1)^2}.
	\end{align*}
	\end{theorem}

	\begin{proof}
	There are $\frac{m(m-1)}{2}$ edges with endpoints of degree $m+n-2$. Also there are $m(n-1)$ edges with endpoints of degree $m+n-2$ and $n-1$ and there are $m(n-1)(\frac{n}{2}-1)$ edges with endpoints of degree $n-1$. Therefore 
	\begin{align*}
	SO(Q(m,n))&=\frac{m(m-1)}{2}\sqrt{(m+n-2)^2+(m+n-2)^2}\\
	&\quad+m(n-1)\sqrt{(m+n-2)^2+(n-1)^2}\\
	&\quad +m(n-1)(\frac{n}{2}-1)\sqrt{(n-1)^2+(n-1)^2},
	\end{align*}
	and we have the result.	
	\qed
	\end{proof}

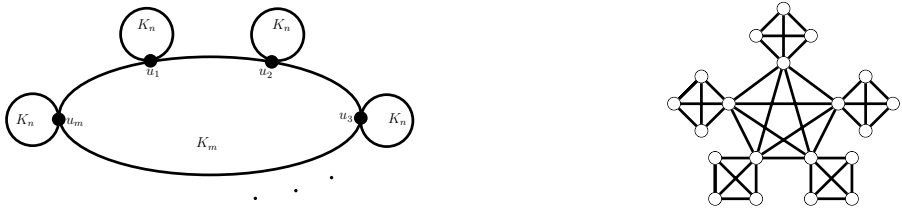
\begin{figure}\hspace{1.02cm}
	\begin{minipage}{7.5cm}
		\psscalebox{0.45 0.45}
{
\begin{pspicture}(0,-13.324639)(11.946668,-7.6353607)
\pscircle[linecolor=black, linewidth=0.08, dimen=outer](4.1333337,-8.435361){0.8}
\pscircle[linecolor=black, linewidth=0.08, dimen=outer](11.146667,-10.982027){0.8}
\pscircle[linecolor=black, linewidth=0.08, dimen=outer](7.96,-8.448694){0.8}
\pscircle[linecolor=black, linewidth=0.08, dimen=outer](0.8000002,-10.95536){0.8}
\psdots[linecolor=black, dotsize=0.4](4.2400002,-9.222028)
\psdots[linecolor=black, dotsize=0.4](10.386667,-10.902027)
\psdots[linecolor=black, dotsize=0.4](7.786667,-9.248694)
\psdots[linecolor=black, dotsize=0.4](1.5600002,-10.942027)
\rput[bl](5.5733337,-11.835361){$\Large{K_m}$}
\rput[bl](0.30666688,-11.142028){$\large{K_n}$}
\rput[bl](3.8533335,-8.342028){$\large{K_n}$}
\rput[bl](7.8,-8.355361){$\large{K_n}$}
\rput[bl](11.1866665,-11.11536){$\large{K_n}$}
\psellipse[linecolor=black, linewidth=0.08, dimen=outer](5.9866667,-10.842028)(4.4533334,1.78)
\psdots[linecolor=black, dotsize=0.1](9.533334,-12.662027)
\psdots[linecolor=black, dotsize=0.1](8.493334,-13.048694)
\psdots[linecolor=black, dotsize=0.1](7.306667,-13.275361)
\rput[bl](4.1200004,-9.688694){$u_1$}
\rput[bl](7.4133334,-9.715361){$u_2$}
\rput[bl](9.76,-11.008694){$u_3$}
\rput[bl](1.7733335,-11.142028){$u_m$}
\end{pspicture}
}
	\end{minipage}
	\hspace{1.02cm}
	\begin{minipage}{7.5cm} 
		\psscalebox{0.45 0.45}
		{
			\begin{pspicture}(0,-4.8)(6.8027782,1.202778)
			\psline[linecolor=black, linewidth=0.08](3.4013891,-0.5986108)(1.8013892,-1.7986108)(2.6013892,-3.3986108)(4.2013893,-3.3986108)(5.001389,-1.7986108)(3.4013891,-0.5986108)(3.4013891,-0.5986108)
			\psline[linecolor=black, linewidth=0.08](3.4013891,-0.5986108)(2.6013892,0.20138916)(3.4013891,1.0013891)(4.2013893,0.20138916)(3.4013891,-0.5986108)(3.4013891,-0.5986108)
			\psline[linecolor=black, linewidth=0.08](3.4013891,1.0013891)(3.4013891,-0.5986108)(3.4013891,-0.5986108)
			\psline[linecolor=black, linewidth=0.08](2.6013892,0.20138916)(4.2013893,0.20138916)(4.2013893,0.20138916)
			\psline[linecolor=black, linewidth=0.08](5.001389,-1.7986108)(5.801389,-0.99861085)(6.601389,-1.7986108)(5.801389,-2.5986109)(5.001389,-1.7986108)(6.601389,-1.7986108)(6.601389,-1.7986108)
			\psline[linecolor=black, linewidth=0.08](5.801389,-0.99861085)(5.801389,-2.5986109)(5.801389,-2.5986109)
			\psline[linecolor=black, linewidth=0.08](4.2013893,-3.3986108)(5.401389,-3.3986108)(5.401389,-4.598611)(4.2013893,-4.598611)(4.2013893,-3.3986108)(5.401389,-4.598611)(5.401389,-4.598611)
			\psline[linecolor=black, linewidth=0.08](5.401389,-3.3986108)(4.2013893,-4.598611)(4.2013893,-4.598611)
			\psline[linecolor=black, linewidth=0.08](2.6013892,-3.3986108)(2.6013892,-4.598611)(1.4013891,-4.598611)(1.4013891,-3.3986108)(2.6013892,-3.3986108)(1.4013891,-4.598611)(1.4013891,-3.3986108)(2.6013892,-4.598611)(2.6013892,-4.598611)
			\psline[linecolor=black, linewidth=0.08](1.8013892,-1.7986108)(1.0013891,-0.99861085)(0.20138916,-1.7986108)(1.0013891,-2.5986109)(1.8013892,-1.7986108)(0.20138916,-1.7986108)(1.0013891,-0.99861085)(1.0013891,-2.5986109)(1.0013891,-2.5986109)
			\psline[linecolor=black, linewidth=0.08](3.4013891,-0.5986108)(2.6013892,-3.3986108)(5.001389,-1.7986108)(1.8013892,-1.7986108)(4.2013893,-3.3986108)(3.4013891,-0.5986108)(3.4013891,-0.5986108)
			\psdots[linecolor=black, dotstyle=o, dotsize=0.4, fillcolor=white](3.4013891,-0.5986108)
			\psdots[linecolor=black, dotstyle=o, dotsize=0.4, fillcolor=white](4.2013893,0.20138916)
			\psdots[linecolor=black, dotstyle=o, dotsize=0.4, fillcolor=white](3.4013891,1.0013891)
			\psdots[linecolor=black, dotstyle=o, dotsize=0.4, fillcolor=white](2.6013892,0.20138916)
			\psdots[linecolor=black, dotstyle=o, dotsize=0.4, fillcolor=white](5.801389,-0.99861085)
			\psdots[linecolor=black, dotstyle=o, dotsize=0.4, fillcolor=white](5.001389,-1.7986108)
			\psdots[linecolor=black, dotstyle=o, dotsize=0.4, fillcolor=white](5.801389,-2.5986109)
			\psdots[linecolor=black, dotstyle=o, dotsize=0.4, fillcolor=white](6.601389,-1.7986108)
			\psdots[linecolor=black, dotstyle=o, dotsize=0.4, fillcolor=white](1.8013892,-1.7986108)
			\psdots[linecolor=black, dotstyle=o, dotsize=0.4, fillcolor=white](1.0013891,-2.5986109)
			\psdots[linecolor=black, dotstyle=o, dotsize=0.4, fillcolor=white](0.20138916,-1.7986108)
			\psdots[linecolor=black, dotstyle=o, dotsize=0.4, fillcolor=white](1.0013891,-0.99861085)
			\psdots[linecolor=black, dotstyle=o, dotsize=0.4, fillcolor=white](1.4013891,-3.3986108)
			\psdots[linecolor=black, dotstyle=o, dotsize=0.4, fillcolor=white](1.4013891,-4.598611)
			\psdots[linecolor=black, dotstyle=o, dotsize=0.4, fillcolor=white](2.6013892,-3.3986108)
			\psdots[linecolor=black, dotstyle=o, dotsize=0.4, fillcolor=white](2.6013892,-4.598611)
			\psdots[linecolor=black, dotstyle=o, dotsize=0.4, fillcolor=white](4.2013893,-3.3986108)
			\psdots[linecolor=black, dotstyle=o, dotsize=0.4, fillcolor=white](5.401389,-3.3986108)
			\psdots[linecolor=black, dotstyle=o, dotsize=0.4, fillcolor=white](5.401389,-4.598611)
			\psdots[linecolor=black, dotstyle=o, dotsize=0.4, fillcolor=white](4.2013893,-4.598611)
			\end{pspicture}
		}
	\end{minipage}
	\caption{The graph $Q(m,n)$ and $Q(5,4)$, respectively. }\label{qmn}
\end{figure}

To obtain more results,  we need the following theorem:

\begin{theorem}\label{pro(G-e)}{\rm\cite{Sombor}}
	 Let $G=(V,E)$ be a graph and $e=uv\in E$. Also let $d_w$ be the degree of vertex $w$ in $G$. Then,
	$$SO(G-e) < SO(G) - \frac{|d_u-d_v|}{\sqrt{2}}. $$
\end{theorem}

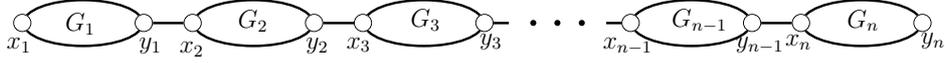
\begin{figure}
	\begin{center}
		\psscalebox{0.8 0.8}
		{
			\begin{pspicture}(0,-4.08)(15.436667,-3.12)
			\psellipse[linecolor=black, linewidth=0.04, dimen=outer](1.2533334,-3.52)(1.0,0.4)
			\psellipse[linecolor=black, linewidth=0.04, dimen=outer](4.0533333,-3.52)(1.0,0.4)
			\psellipse[linecolor=black, linewidth=0.04, dimen=outer](6.853334,-3.52)(1.0,0.4)
			\psellipse[linecolor=black, linewidth=0.04, dimen=outer](11.253334,-3.52)(1.0,0.4)
			\psellipse[linecolor=black, linewidth=0.04, dimen=outer](14.053333,-3.52)(1.0,0.4)
			\psline[linecolor=black, linewidth=0.04](2.2533333,-3.52)(3.0533335,-3.52)(3.0533335,-3.52)
			\psline[linecolor=black, linewidth=0.04](5.0533333,-3.52)(5.8533335,-3.52)(5.8533335,-3.52)
			\psline[linecolor=black, linewidth=0.04](12.253333,-3.52)(13.053333,-3.52)(13.053333,-3.52)
			\psdots[linecolor=black, dotstyle=o, dotsize=0.3, fillcolor=white](2.2533333,-3.52)
			\psdots[linecolor=black, dotstyle=o, dotsize=0.3, fillcolor=white](0.25333345,-3.52)
			\psdots[linecolor=black, dotstyle=o, dotsize=0.3, fillcolor=white](3.0533335,-3.52)
			\psdots[linecolor=black, dotstyle=o, dotsize=0.3, fillcolor=white](5.0533333,-3.52)
			\psdots[linecolor=black, dotstyle=o, dotsize=0.3, fillcolor=white](5.8533335,-3.52)
			\psdots[linecolor=black, dotstyle=o, dotsize=0.3, fillcolor=white](12.253333,-3.52)
			\psdots[linecolor=black, dotstyle=o, dotsize=0.3, fillcolor=white](13.053333,-3.52)
			\psdots[linecolor=black, dotstyle=o, dotsize=0.3, fillcolor=white](15.053333,-3.52)
			\rput[bl](0.0,-4.0133333){$x_1$}
			\rput[bl](2.8400002,-4.08){$x_2$}
			\rput[bl](5.5866666,-4.0){$x_3$}
			\rput[bl](2.1733334,-4.0266666){$y_1$}
			\rput[bl](4.92,-4.0266666){$y_2$}
			\rput[bl](7.7733335,-3.96){$y_3$}
			\rput[bl](0.9600001,-3.7066667){$G_1$}
			\rput[bl](3.8000002,-3.6666667){$G_2$}
			\rput[bl](6.64,-3.64){$G_3$}
			\psline[linecolor=black, linewidth=0.04](8.253333,-3.52)(7.8533335,-3.52)(7.8533335,-3.52)
			\psline[linecolor=black, linewidth=0.04](9.853333,-3.52)(10.253333,-3.52)(10.253333,-3.52)
			\psdots[linecolor=black, dotstyle=o, dotsize=0.3, fillcolor=white](7.8533335,-3.52)
			\psdots[linecolor=black, dotstyle=o, dotsize=0.3, fillcolor=white](10.253333,-3.52)
			\psdots[linecolor=black, dotsize=0.1](8.653334,-3.52)
			\psdots[linecolor=black, dotsize=0.1](9.053333,-3.52)
			\psdots[linecolor=black, dotsize=0.1](9.453334,-3.52)
			\rput[bl](12.8,-3.96){$x_n$}
			\rput[bl](15.026667,-3.9466667){$y_n$}
			\rput[bl](11.986667,-4.0266666){$y_{n-1}$}
			\rput[bl](10.933333,-3.68){$G_{n-1}$}
			\rput[bl](9.8,-4.04){$x_{n-1}$}
			\rput[bl](13.8133335,-3.6533334){$G_n$}
			\end{pspicture}
		}
	\end{center}
	\caption{Link of $n$ graphs $G_1,G_2, \ldots , G_n$} \label{link-n}
\end{figure}

Here we study the Sombor index for links of graphs, circuits of
graphs, chains of graphs, and bouquets of graphs. 

\begin{theorem} \label{thm-link}
 Let $G$ be a polymer graph with composed of 	monomers $\{G_i\}_{i=1}^k$ with respect to the vertices $\{x_i, y_i\}_{i=1}^k$. Let $G$ be the link of graphs  (see Figure \ref{link-n}).  Then,
	$$SO(G)>\sum_{i=1}^{k}SO(G_i)+\sum_{i=1}^{k-1}\frac{|d_{x_{i+1}}-d_{y_i}|}{\sqrt{2}}.$$
\end{theorem}

\begin{proof}
	First we remove edge $y_1x_2$ (Figure \ref{link-n}). By Proposition \ref{pro(G-e)}, we have
	$$SO(G) > SO(G-y_1x_2) + \frac{|d_{y_1}-d_{x_2}|}{\sqrt{2}}.$$
	Let $G^{\prime}$ be the link graph related to graphs $\{G_i\}_{i=2}^k$ with respect to the vertices $\{x_i, y_i\}_{i=2}^k$. Then we have,
	$$SO(G-y_1x_2)=SO(G_1)+SO(G^{\prime}),$$
	and therefore,
	$$SO(G) > SO(G_1)+SO(G^{\prime}) + \frac{|d_{y_1}-d_{x_2}|}{\sqrt{2}}.$$
	By continuing this process, we have the result.
	\qed
\end{proof}

\begin{theorem} 
	Let $G_1,G_2, \ldots , G_k$ be a finite sequence of pairwise disjoint connected graphs and let
	$x_i \in V(G_i)$. Let $G$ be the circuit of graphs $\{G_i\}_{i=1}^k$ with respect to the vertices $\{x_i\}_{i=1}^k$ and obtained by identifying the vertex $x_i$ of the graph $G_i$ with the $i$-th vertex of the
	cycle graph $C_k$ (Figure \ref{circuit-n}). Then,
	$$SO(G)>\frac{|d_{x_{1}}-d_{x_{n}}|}{\sqrt{2}}+\sum_{i=1}^{k}SO(G_i)+\sum_{i=1}^{k-1}\frac{|d_{x_{i}}-d_{x_{i+1}}|}{\sqrt{2}}.$$
\end{theorem}

\begin{proof}
	First we remove edge $x_nx_1$ (Figure \ref{circuit-n}). By Proposition \ref{pro(G-e)}, we have
	$$SO(G) > SO(G-x_nx_1) + \frac{|d_{x_n}-d_{x_1}|}{\sqrt{2}}.$$
	Now we remove edge $x_1x_2$. Then,
	$$SO(G) > SO(G-\{x_nx_1,x_1x_2\}) + \frac{|d_{x_n}-d_{x_1}|}{\sqrt{2}}+\frac{|d_{x_2}-d_{x_1}|}{\sqrt{2}}.$$
	Let $G^{\prime}$ be the graph related to circuit graph with $\{G_i\}_{i=2}^k$ with respect to the vertices $\{x_i\}_{i=2}^k$ and removing the edge $x_nx_1$. Then  we have,
	$$SO(G-\{x_nx_1,x_1x_2\})=SO(G_1)+SO(G^{\prime}),$$
	and therefore,
	$$SO(G) > SO(G_1)+SO(G^{\prime}) + \frac{|d_{x_n}-d_{x_1}|}{\sqrt{2}}+\frac{|d_{x_2}-d_{x_1}|}{\sqrt{2}}.$$
	By continuing this process, we have the result.
	\qed
\end{proof}

\begin{figure}
	\begin{center}
		\psscalebox{0.85 0.85}
		{
			\begin{pspicture}(0,-7.6)(5.6,-2.0)
			\rput[bl](2.6533334,-4.48){$x_1$}
			\rput[bl](3.0533333,-4.92){$x_2$}
			\rput[bl](2.5733333,-5.4266667){$x_3$}
			\rput[bl](2.6,-3.1733334){$G_1$}
			\rput[bl](4.2933335,-4.9866667){$G_2$}
			\rput[bl](2.6133332,-6.7733335){$G_3$}
			\rput[bl](2.1733334,-4.9466667){$x_n$}
			\rput[bl](0.73333335,-4.9333334){$G_n$}
			\psellipse[linecolor=black, linewidth=0.04, dimen=outer](1.0,-4.8)(1.0,0.4)
			\psellipse[linecolor=black, linewidth=0.04, dimen=outer](4.6,-4.8)(1.0,0.4)
			\psellipse[linecolor=black, linewidth=0.04, dimen=outer](2.8,-3.0)(0.4,1.0)
			\psellipse[linecolor=black, linewidth=0.04, dimen=outer](2.8,-6.6)(0.4,1.0)
			\psline[linecolor=black, linewidth=0.04](2.0,-4.8)(2.8,-4.0)(3.6,-4.8)(2.8,-5.6)(2.8,-5.6)
			\psdots[linecolor=black, fillstyle=solid, dotstyle=o, dotsize=0.3, fillcolor=white](2.8,-4.0)
			\psdots[linecolor=black, fillstyle=solid, dotstyle=o, dotsize=0.3, fillcolor=white](3.6,-4.8)
			\psline[linecolor=black, linewidth=0.04, linestyle=dotted, dotsep=0.10583334cm](2.8,-5.6)(2.0,-4.8)(2.0,-4.8)
			\psdots[linecolor=black, fillstyle=solid, dotstyle=o, dotsize=0.3, fillcolor=white](2.0,-4.8)
			\psdots[linecolor=black, fillstyle=solid, dotstyle=o, dotsize=0.3, fillcolor=white](2.8,-5.6)
			\end{pspicture}
		}
	\end{center}
	\caption{Circuit of $n$ graphs $G_1,G_2, \ldots , G_n$} \label{circuit-n}
\end{figure}
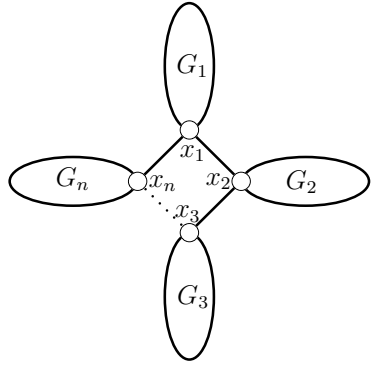

The following theorem is another lower bound for the Sombor index of the circuit of graphs. 

\begin{theorem} 
	Let $G_1,G_2, \ldots , G_k$ be a finite sequence of pairwise disjoint connected graphs and let
	$x_i \in V(G_i)$. Let $G$ be the circuit of graphs $\{G_i\}_{i=1}^k$ with respect to the vertices $\{x_i\}_{i=1}^k$ and obtained by identifying the vertex $x_i$ of the graph $G_i$ with the $i$-th vertex of the
	cycle graph $C_k$ (Figure \ref{circuit-n}). Then,
	$$SO(G)\geq 2k\sqrt{2}+ \sum_{i=1}^{k}SO(G_i).$$
	The equality holds if and only if for every $1\leq i\leq k$, $G_i=K_1$.
\end{theorem}
\begin{proof}
	Let $d_i$ be the degree of the vertex $x_i$ before creating $G$. Since $d(x_i)=d_i+2$, we have:
	\begin{align*}
	SO(G)&=\sqrt{(d_k+2)^2+(d_{1}+2)^2}+\sum_{i=1}^{k-1}\sqrt{(d_i+2)^2+(d_{i+1}+2)^2}\\
	&\quad+\sum_{i=1}^{k}\big(\sum_{uv\in E(G_i-x_i)}\sqrt{d_u^2+d_v^2}+\sum_{x_i\sim u\in G_i}\sqrt{(d_i+2)^2+d_u^2}~~ \big)\\
	&\geq \sqrt{4+4}+\sum_{i=1}^{k-1}\sqrt{4+4}+\sum_{i=1}^{k}\big(\sum_{uv\in E(G_i-x_i)}\sqrt{d_u^2+d_v^2}+\sum_{x_i\sim u\in G_i}\sqrt{d_i^2+d_u^2} ~~\big)\\
	&=2k\sqrt{2}+\sum_{i=1}^{k}SO(G_i).
	\end{align*}
	If $G_i$ has at least one edge then the equality does not hold and therefore we have the result.
	\qed
\end{proof}			

\begin{figure}
	\begin{center}
		\psscalebox{0.6 0.6}
		{
			\begin{pspicture}(0,-3.9483333)(12.236668,-2.8316667)
			\psellipse[linecolor=black, linewidth=0.04, dimen=outer](1.2533334,-3.4416668)(1.0,0.4)
			\psellipse[linecolor=black, linewidth=0.04, dimen=outer](3.2533333,-3.4416668)(1.0,0.4)
			\psellipse[linecolor=black, linewidth=0.04, dimen=outer](5.2533336,-3.4416668)(1.0,0.4)
			\psellipse[linecolor=black, linewidth=0.04, dimen=outer](8.853333,-3.4416668)(1.0,0.4)
			\psellipse[linecolor=black, linewidth=0.04, dimen=outer](10.853333,-3.4416668)(1.0,0.4)
			\psdots[linecolor=black, fillstyle=solid, dotstyle=o, dotsize=0.3, fillcolor=white](2.2533333,-3.4416666)
			\psdots[linecolor=black, fillstyle=solid, dotstyle=o, dotsize=0.3, fillcolor=white](0.25333345,-3.4416666)
			\psdots[linecolor=black, fillstyle=solid, dotstyle=o, dotsize=0.3, fillcolor=white](2.2533333,-3.4416666)
			\psdots[linecolor=black, fillstyle=solid, dotstyle=o, dotsize=0.3, fillcolor=white](4.2533336,-3.4416666)
			\psdots[linecolor=black, fillstyle=solid, dotstyle=o, dotsize=0.3, fillcolor=white](4.2533336,-3.4416666)
			\psdots[linecolor=black, fillstyle=solid, dotstyle=o, dotsize=0.3, fillcolor=white](9.853333,-3.4416666)
			\psdots[linecolor=black, fillstyle=solid, dotstyle=o, dotsize=0.3, fillcolor=white](9.853333,-3.4416666)
			\psdots[linecolor=black, fillstyle=solid, dotstyle=o, dotsize=0.3, fillcolor=white](11.853333,-3.4416666)
			\rput[bl](0.0,-3.135){$x_1$}
			\rput[bl](2.0400002,-3.2016668){$x_2$}
			\rput[bl](3.9866667,-3.1216667){$x_3$}
			\rput[bl](2.1733334,-3.9483335){$y_1$}
			\rput[bl](4.12,-3.9483335){$y_2$}
			\rput[bl](6.1733336,-3.8816667){$y_3$}
			\rput[bl](0.9600001,-3.6283333){$G_1$}
			\rput[bl](3.0,-3.5883334){$G_2$}
			\rput[bl](5.04,-3.5616667){$G_3$}
			\psdots[linecolor=black, fillstyle=solid, dotstyle=o, dotsize=0.3, fillcolor=white](6.2533336,-3.4416666)
			\psdots[linecolor=black, fillstyle=solid, dotstyle=o, dotsize=0.3, fillcolor=white](7.8533335,-3.4416666)
			\psdots[linecolor=black, dotsize=0.1](6.6533337,-3.4416666)
			\psdots[linecolor=black, dotsize=0.1](7.0533333,-3.4416666)
			\psdots[linecolor=black, dotsize=0.1](7.4533334,-3.4416666)
			\rput[bl](9.6,-3.0816667){$x_n$}
			\rput[bl](11.826667,-3.8683333){$y_n$}
			\rput[bl](9.586667,-3.9483335){$y_{n-1}$}
			\rput[bl](8.533334,-3.6016667){$G_{n-1}$}
			\rput[bl](7.4,-3.1616666){$x_{n-1}$}
			\rput[bl](10.613334,-3.575){$G_n$}
			\end{pspicture}
		}
	\end{center}
	\caption{Chain of $n$ graphs $G_1,G_2, \ldots , G_n$} \label{chain-n}
\end{figure}
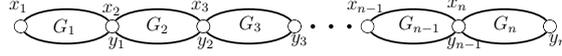

\begin{theorem} \label{thm-chain}
	Let $G_1,G_2, \ldots , G_n$ be a finite sequence of pairwise disjoint connected graphs and let
	$x_i,y_i \in V(G_i)$. Let $C(G_1,...,G_n)$ be the chain of graphs $\{G_i\}_{i=1}^n$ with respect to the vertices $\{x_i, y_i\}_{i=1}^k$ which obtained by identifying the vertex $y_i$ with the vertex $x_{i+1}$ for $i=1,2,\ldots,n-1$ (Figure \ref{chain-n}). Then,
\begin{enumerate}
	\item[(i)]  $$SO(C(G_1,...,G_n))>SO(C(G_1,...,G_{n-1}))+SO(G_n-y_{n-1})+\displaystyle\sum_{\substack{u\sim y_{n-1}\\u\in V(G_n)}} \frac{|d_u-d_{y_{n-1}}|}{\sqrt{2}}.$$

	\item[(ii)] 
		 $$SO(C(G_1,...,G_n))>SO(C(G_{1}))+\sum_{i=2}^n SO(G_i-y_{i-1})+\displaystyle\sum_{i=1}^{n-1}\sum_{\substack{u\sim y_{i}\\ u\in V(G_{i+1})}} \frac{|d_u-d_{y_{i}}|}{\sqrt{2}}.$$
		\end{enumerate}
\end{theorem}
\begin{proof}
	\begin{enumerate}
		\item[(i)] 
	Consider $C(G_1,...,G_n)$ in Figure \ref{chain-n}. By using inductively  Theorem \ref{pro(G-e)} for all edges in $G_n$ which one of the their end vertices is $y_{n-1}$ we have the result. 
	
	\item[(ii)] 
	The result follows by applying Part (i) inductively.  	
		\qed
		\end{enumerate} 
\end{proof}

Similar to the Theorem \ref{thm-chain} we have:

\begin{theorem} 
	Let $G_1,G_2, \ldots , G_n$ be a finite sequence of pairwise disjoint connected graphs and let
	$x_i \in V(G_i)$. Let $B(G_1,...,G_n)$ be the bouquet of graphs $\{G_i\}_{i=1}^n$ with respect to the vertices $\{x_i\}_{i=1}^n$ and obtained by identifying the vertex $x_i$ of the graph $G_i$ with $x$ (see Figure \ref{bouquet-n}). Then,	
	$$SO(B(G_1,...,G_n))>SO(G_{1})+\sum_{i=2}^n SO(G_i-x_{i})+\sum_{i=1}^{n-1}\sum_{\substack{u\sim x_{i+1}\\u\in V(G_{i+1})}} \frac{|d_u-d_{x_{i+1}}|}{\sqrt{2}}.$$
\end{theorem}

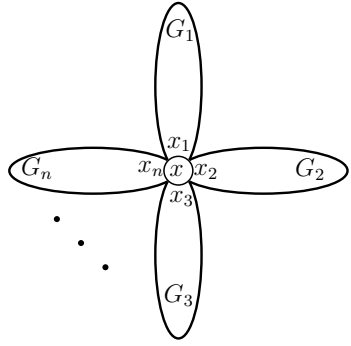
\begin{figure}
	\begin{center}
		\psscalebox{0.8 0.8}
		{
			\begin{pspicture}(0,-6.76)(5.6,-1.16)
			\rput[bl](2.6133332,-3.64){$x_1$}
			\rput[bl](3.0533333,-4.0933332){$x_2$}
			\rput[bl](2.6533334,-4.5466666){$x_3$}
			\rput[bl](2.5866666,-1.8133334){$G_1$}
			\rput[bl](4.72,-4.1066666){$G_2$}
			\rput[bl](2.56,-6.2){$G_3$}
			\rput[bl](2.1333334,-4.04){$x_n$}
			\rput[bl](0.21333334,-4.0933332){$G_n$}
			\psellipse[linecolor=black, linewidth=0.04, dimen=outer](1.4,-3.96)(1.4,0.4)
			\psellipse[linecolor=black, linewidth=0.04, dimen=outer](2.8,-2.56)(0.4,1.4)
			\psellipse[linecolor=black, linewidth=0.04, dimen=outer](4.2,-3.96)(1.4,0.4)
			\psellipse[linecolor=black, linewidth=0.04, dimen=outer](2.8,-5.36)(0.4,1.4)
			\psdots[linecolor=black, dotsize=0.1](0.8,-4.76)
			\psdots[linecolor=black, dotsize=0.1](1.2,-5.16)
			\psdots[linecolor=black, dotsize=0.1](1.6,-5.56)
			\psdots[linecolor=black, dotstyle=o, dotsize=0.5, fillcolor=white](2.8,-3.96)
			\rput[bl](2.6533334,-4.04){$x$}
			\end{pspicture}
		}
	\end{center}
	\caption{Bouquet of $n$ graphs $G_1,G_2, \ldots , G_n$ and $x_1=x_2=\ldots=x_n=x$} \label{bouquet-n}
\end{figure}

\section{Chemical applications}

In this section, we apply our previous results in order to obtain the Sombor index  of
families of graphs that are of importance in chemistry.

    \subsection{Spiro-chains}
    Spiro-chains are defined in \cite{Diudea}. Making use of the concept of chain of graphs, a
    spiro-chain can be defined as a chain of cycles. We denote by $S_{q,h,k}$ the chain of $k$ cycles $C_q$ in which the distance between two consecutive contact vertices is $h$ (see $S_{6,2,8}$ in Figure \ref{S628}). 
    
    \begin{theorem} \label{thm-l(qhk)}
    	For the graph  $S_{q,h,k}$, when $h\geq 2$, we have:
    	\begin{align*}
    	SO(S_{q,h,k})=(2qk-8k+8)\sqrt{2}+(8k-8)\sqrt{5}.
    	\end{align*}
    \end{theorem}

    \begin{proof}
    	There are $4(k-1)$ edges with endpoints of degree 2 and 4. Also  there are $qk-4(k-1)$ edges with endpoints of degree 2. Therefore 
    	\begin{align*}
    	SO(S_{q,h,k})=4(k-1)\sqrt{4+16}+(qk-4(k-1))\sqrt{4+4},
    	\end{align*}
    	and we have the result.	
    	\qed
    \end{proof}

    \begin{theorem} \label{thm-l(q1k)}
    	For the graph  $S_{q,1,k}$, we have:
    	\begin{align*}
    	SO(S_{q,1,k})=(2qk-2k-4)\sqrt{2}+4k\sqrt{5}.
    	\end{align*}
    \end{theorem}

    \begin{proof}
    	There are $k-2$ edges with endpoints of degree 4. Also there are $2k$ edges with endpoints of degree 4 and 2, and there are $qk-3k+2$ edges with endpoints of degree 2. Therefore 
    	\begin{align*}
    	SO(S_{q,1,k})=(k-2)\sqrt{16+16}+2k\sqrt{16+4}+ (qk-3k+2)\sqrt{4+4},
    	\end{align*}
    	and we have the result.	
    	\qed
    \end{proof}	
    
    \begin{figure}
    	\begin{center}
    		\psscalebox{0.45 0.45}
    		{
    			\begin{pspicture}(0,-5.2)(16.394232,-0.0057690428)
    			\psline[linecolor=black, linewidth=0.08](0.9971153,-2.6028845)(2.1971154,-2.6028845)(2.9971154,-3.8028846)(2.1971154,-5.0028844)(0.9971153,-5.0028844)(0.19711533,-3.8028846)(0.9971153,-2.6028845)(0.9971153,-2.6028845)
    			\psline[linecolor=black, linewidth=0.08](2.1971154,-0.20288453)(3.3971152,-0.20288453)(4.1971154,-1.4028845)(3.3971152,-2.6028845)(2.1971154,-2.6028845)(1.3971153,-1.4028845)(2.1971154,-0.20288453)(2.1971154,-0.20288453)
    			\psline[linecolor=black, linewidth=0.08](4.997115,-0.20288453)(6.1971154,-0.20288453)(6.997115,-1.4028845)(6.1971154,-2.6028845)(4.997115,-2.6028845)(4.1971154,-1.4028845)(4.997115,-0.20288453)(4.997115,-0.20288453)
    			\psline[linecolor=black, linewidth=0.08](6.1971154,-2.6028845)(7.397115,-2.6028845)(8.197115,-3.8028846)(7.397115,-5.0028844)(6.1971154,-5.0028844)(5.397115,-3.8028846)(6.1971154,-2.6028845)(6.1971154,-2.6028845)
    			\psline[linecolor=black, linewidth=0.08](8.997115,-2.6028845)(10.197115,-2.6028845)(10.997115,-3.8028846)(10.197115,-5.0028844)(8.997115,-5.0028844)(8.197115,-3.8028846)(8.997115,-2.6028845)(8.997115,-2.6028845)
    			\psline[linecolor=black, linewidth=0.08](10.197115,-0.20288453)(11.397116,-0.20288453)(12.197115,-1.4028845)(11.397116,-2.6028845)(10.197115,-2.6028845)(9.397116,-1.4028845)(10.197115,-0.20288453)(10.197115,-0.20288453)
    			\psline[linecolor=black, linewidth=0.08](12.997115,-0.20288453)(14.197115,-0.20288453)(14.997115,-1.4028845)(14.197115,-2.6028845)(12.997115,-2.6028845)(12.197115,-1.4028845)(12.997115,-0.20288453)(12.997115,-0.20288453)
    			\psline[linecolor=black, linewidth=0.08](14.197115,-2.6028845)(15.397116,-2.6028845)(16.197115,-3.8028846)(15.397116,-5.0028844)(14.197115,-5.0028844)(13.397116,-3.8028846)(14.197115,-2.6028845)(14.197115,-2.6028845)
    			\psdots[linecolor=black, dotsize=0.4](14.197115,-2.6028845)
    			\psdots[linecolor=black, dotsize=0.4](13.397116,-3.8028846)
    			\psdots[linecolor=black, dotsize=0.4](14.197115,-5.0028844)
    			\psdots[linecolor=black, dotsize=0.4](15.397116,-5.0028844)
    			\psdots[linecolor=black, dotsize=0.4](16.197115,-3.8028846)
    			\psdots[linecolor=black, dotsize=0.4](15.397116,-2.6028845)
    			\psdots[linecolor=black, dotsize=0.4](14.997115,-1.4028845)
    			\psdots[linecolor=black, dotsize=0.4](14.197115,-0.20288453)
    			\psdots[linecolor=black, dotsize=0.4](12.997115,-0.20288453)
    			\psdots[linecolor=black, dotsize=0.4](12.197115,-1.4028845)
    			\psdots[linecolor=black, dotsize=0.4](12.997115,-2.6028845)
    			\psdots[linecolor=black, dotsize=0.4](11.397116,-2.6028845)
    			\psdots[linecolor=black, dotsize=0.4](10.197115,-2.6028845)
    			\psdots[linecolor=black, dotsize=0.4](9.397116,-1.4028845)
    			\psdots[linecolor=black, dotsize=0.4](10.197115,-0.20288453)
    			\psdots[linecolor=black, dotsize=0.4](11.397116,-0.20288453)
    			\psdots[linecolor=black, dotsize=0.4](10.997115,-3.8028846)
    			\psdots[linecolor=black, dotsize=0.4](10.197115,-5.0028844)
    			\psdots[linecolor=black, dotsize=0.4](8.997115,-2.6028845)
    			\psdots[linecolor=black, dotsize=0.4](8.197115,-3.8028846)
    			\psdots[linecolor=black, dotsize=0.4](8.997115,-5.0028844)
    			\psdots[linecolor=black, dotsize=0.4](7.397115,-5.0028844)
    			\psdots[linecolor=black, dotsize=0.4](6.1971154,-5.0028844)
    			\psdots[linecolor=black, dotsize=0.4](5.397115,-3.8028846)
    			\psdots[linecolor=black, dotsize=0.4](6.1971154,-2.6028845)
    			\psdots[linecolor=black, dotsize=0.4](7.397115,-2.6028845)
    			\psdots[linecolor=black, dotsize=0.4](6.997115,-1.4028845)
    			\psdots[linecolor=black, dotsize=0.4](6.1971154,-0.20288453)
    			\psdots[linecolor=black, dotsize=0.4](4.997115,-0.20288453)
    			\psdots[linecolor=black, dotsize=0.4](4.1971154,-1.4028845)
    			\psdots[linecolor=black, dotsize=0.4](4.997115,-2.6028845)
    			\psdots[linecolor=black, dotsize=0.4](3.3971152,-2.6028845)
    			\psdots[linecolor=black, dotsize=0.4](3.3971152,-0.20288453)
    			\psdots[linecolor=black, dotsize=0.4](2.1971154,-0.20288453)
    			\psdots[linecolor=black, dotsize=0.4](1.3971153,-1.4028845)
    			\psdots[linecolor=black, dotsize=0.4](2.1971154,-2.6028845)
    			\psdots[linecolor=black, dotsize=0.4](0.9971153,-2.6028845)
    			\psdots[linecolor=black, dotsize=0.4](0.19711533,-3.8028846)
    			\psdots[linecolor=black, dotsize=0.4](0.9971153,-5.0028844)
    			\psdots[linecolor=black, dotsize=0.4](2.1971154,-5.0028844)
    			\psdots[linecolor=black, dotsize=0.4](2.9971154,-3.8028846)
    			\end{pspicture}
    		}
    	\end{center}
    	\caption{The graph $S_{6,2,8}$}\label{S628}
    \end{figure}
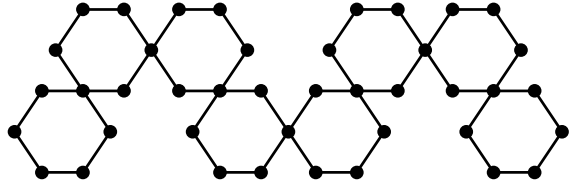

  Cactus graphs which are a class of simple linear polymers, were first known as Husimi tree, they appeared in the scientific literature some sixty years ago in papers by Husimi and
  Riddell concerned with cluster integrals in the theory of condensation in statistical mechanics \cite{9,12,14}. 
  We refer the reader to papers \cite{chellali,13,Gutindex} for some aspects of parameters of  cactus graphs.

 As an immediate result of Theorems \ref{thm-l(qhk)} and \ref{thm-l(q1k)} we have the following results for cactus chains (see \cite{Sombor}):

 \begin{corollary}\rm{\cite{Sombor}}
 	\begin{enumerate} 
 		\item[(i)] 
 		Let $T_n$ be the chain triangular graph (see Figure \ref{paraChainsqu}) of order $n$. Then for every $n\geq 2$,
 		$SO(T_n)=(4n-4)\sqrt{2}+4n\sqrt{5}.$
 		\item[(ii)] 
 		Let $Q_n$ be the para-chain square cactus graph (see Figure \ref{paraChainsqu}) of order $n$. Then for every $n\geq 2$,
 		$SO(Q_n)=8\sqrt{2}+(8n-8)\sqrt{5}.$
 		\item[(iii)] 
 		Let $O_n$ be the para-chain square cactus (see Figure \ref{ortho-ohn}) graph of order $n$. Then for every $n\geq 2$,
 		$SO(O_n)=(6n-4)\sqrt{2}+4n\sqrt{5}.$
 		\item[(iv)] 
 		Let $O_n^h$ be the Ortho-chain graph (see Figure \ref{ortho-ohn}) of order $n$. Then for every $n\geq 2$,
 		$SO(O_n^h)=(10n-4)\sqrt{2}+4n\sqrt{5}.$
 		\item[(v)] 
 		Let $L_n$ be the para-chain hexagonal
 		cactus graph (see Figure \ref{metaChainMn}) of order $n$. Then for every $n\geq 2$,
 		$SO(L_n)=(4n+8)\sqrt{2}+(8n-8)\sqrt{5}.$
 		\item[(vi)] 
 		Let $M_n$ be the Meta-chain hexagonal
 		cactus graph (see Figure \ref{metaChainMn}) of order $n$. Then for every $n\geq 2$,
 		$SO(M_n)=(4n+8)\sqrt{2}+(8n-8)\sqrt{5}.$
 	\end{enumerate} 
 \end{corollary}

 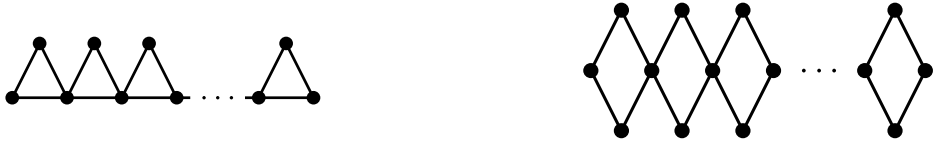
\begin{figure}
 	  		\hspace{1.3cm}	\begin{minipage}{7.5cm}
 				\psscalebox{0.45 0.45}
 				{
 	 			\begin{pspicture}(0,-7.2)(9.194231,-5.205769)
 			\psdots[linecolor=black, dotsize=0.1](5.7971153,-7.0028844)
 			\psdots[linecolor=black, dotsize=0.1](6.1971154,-7.0028844)
 			\psdots[linecolor=black, dotsize=0.1](6.5971155,-7.0028844)
 			\psdots[linecolor=black, dotsize=0.4](7.3971157,-7.0028844)
 			\psdots[linecolor=black, dotsize=0.4](8.197116,-5.4028845)
 			\psdots[linecolor=black, dotsize=0.4](8.997115,-7.0028844)
 			\psdots[linecolor=black, dotsize=0.4](4.9971156,-7.0028844)
 			\psdots[linecolor=black, dotsize=0.4](4.1971154,-5.4028845)
 			\psdots[linecolor=black, dotsize=0.4](3.3971155,-7.0028844)
 			\psdots[linecolor=black, dotsize=0.4](2.5971155,-5.4028845)
 			\psdots[linecolor=black, dotsize=0.4](1.7971154,-7.0028844)
 			\psdots[linecolor=black, dotsize=0.4](0.9971155,-5.4028845)
 			\psdots[linecolor=black, dotsize=0.4](0.19711548,-7.0028844)
 			\psline[linecolor=black, linewidth=0.08](0.19711548,-7.0028844)(4.9971156,-7.0028844)(4.1971154,-5.4028845)(3.3971155,-7.0028844)(2.5971155,-5.4028845)(1.7971154,-7.0028844)(0.9971155,-5.4028845)(0.19711548,-7.0028844)(0.19711548,-7.0028844)
 			\psline[linecolor=black, linewidth=0.08](7.3971157,-7.0028844)(8.197116,-5.4028845)(8.997115,-7.0028844)(7.3971157,-7.0028844)(7.3971157,-7.0028844)
 			\psline[linecolor=black, linewidth=0.08](4.9971156,-7.0028844)(5.3971157,-7.0028844)(5.3971157,-7.0028844)
 			\psline[linecolor=black, linewidth=0.08](7.3971157,-7.0028844)(6.9971156,-7.0028844)(6.9971156,-7.0028844)
 			\end{pspicture}
 		}
 			\end{minipage}
  	\begin{minipage}{7.5cm}
 		\psscalebox{0.5 0.5}
 		{
 			\begin{pspicture}(0,-8.0)(9.194231,-4.405769)
 			\psdots[linecolor=black, dotsize=0.1](5.7971153,-6.2028847)
 			\psdots[linecolor=black, dotsize=0.1](6.1971154,-6.2028847)
 			\psdots[linecolor=black, dotsize=0.1](6.5971155,-6.2028847)
 			\psdots[linecolor=black, dotsize=0.4](7.3971157,-6.2028847)
 			\psdots[linecolor=black, dotsize=0.4](8.197116,-4.6028843)
 			\psdots[linecolor=black, dotsize=0.4](8.997115,-6.2028847)
 			\psdots[linecolor=black, dotsize=0.4](4.9971156,-6.2028847)
 			\psdots[linecolor=black, dotsize=0.4](4.1971154,-4.6028843)
 			\psdots[linecolor=black, dotsize=0.4](3.3971155,-6.2028847)
 			\psdots[linecolor=black, dotsize=0.4](2.5971155,-4.6028843)
 			\psdots[linecolor=black, dotsize=0.4](1.7971154,-6.2028847)
 			\psdots[linecolor=black, dotsize=0.4](0.9971155,-4.6028843)
 			\psdots[linecolor=black, dotsize=0.4](0.19711548,-6.2028847)
 			\psdots[linecolor=black, dotsize=0.4](0.9971155,-7.8028846)
 			\psdots[linecolor=black, dotsize=0.4](2.5971155,-7.8028846)
 			\psdots[linecolor=black, dotsize=0.4](4.1971154,-7.8028846)
 			\psdots[linecolor=black, dotsize=0.4](8.197116,-7.8028846)
 			\psline[linecolor=black, linewidth=0.08](7.3971157,-6.2028847)(8.197116,-4.6028843)(8.997115,-6.2028847)(8.197116,-7.8028846)(7.3971157,-6.2028847)(7.3971157,-6.2028847)
 			\psline[linecolor=black, linewidth=0.08](4.9971156,-6.2028847)(4.1971154,-4.6028843)(3.3971155,-6.2028847)(2.5971155,-4.6028843)(1.7971154,-6.2028847)(0.9971155,-4.6028843)(0.19711548,-6.2028847)(0.9971155,-7.8028846)(1.7971154,-6.2028847)(2.5971155,-7.8028846)(3.3971155,-6.2028847)(4.1971154,-7.8028846)(4.9971156,-6.2028847)(4.9971156,-6.2028847)
 			\end{pspicture}
 		}
 	\end{minipage}
 	 	\caption{Chain triangular cactus $T_n$ and para-chain square cactus $Q_n$} \label{paraChainsqu}
 \end{figure}

 \begin{figure}
 	\hspace{1.3cm}
 	\begin{minipage}{7.5cm}
 		\psscalebox{0.45 0.45}
 		{
 			\begin{pspicture}(0,-6.0)(12.394231,-2.405769)
 			\psdots[linecolor=black, dotsize=0.4](0.19711548,-2.6028845)
 			\psdots[linecolor=black, dotsize=0.4](1.7971154,-2.6028845)
 			\psdots[linecolor=black, dotsize=0.4](0.19711548,-4.2028847)
 			\psdots[linecolor=black, dotsize=0.4](1.7971154,-4.2028847)
 			\psdots[linecolor=black, dotsize=0.4](3.3971155,-4.2028847)
 			\psdots[linecolor=black, dotsize=0.4](4.9971156,-4.2028847)
 			\psdots[linecolor=black, dotsize=0.4](6.5971155,-4.2028847)
 			\psdots[linecolor=black, dotsize=0.4](1.7971154,-5.8028846)
 			\psdots[linecolor=black, dotsize=0.4](3.3971155,-5.8028846)
 			\psdots[linecolor=black, dotsize=0.4](3.3971155,-2.6028845)
 			\psdots[linecolor=black, dotsize=0.4](4.9971156,-2.6028845)
 			\psdots[linecolor=black, dotsize=0.4](6.5971155,-5.8028846)
 			\psdots[linecolor=black, dotsize=0.4](4.9971156,-5.8028846)
 			\psline[linecolor=black, linewidth=0.08](0.19711548,-4.2028847)(6.5971155,-4.2028847)(6.5971155,-4.2028847)
 			\psline[linecolor=black, linewidth=0.08](0.19711548,-2.6028845)(1.7971154,-2.6028845)(1.7971154,-5.8028846)(3.3971155,-5.8028846)(3.3971155,-2.6028845)(4.9971156,-2.6028845)(4.9971156,-5.8028846)(6.5971155,-5.8028846)(6.5971155,-4.2028847)(6.5971155,-4.2028847)
 			\psline[linecolor=black, linewidth=0.08](0.19711548,-4.2028847)(0.19711548,-2.6028845)(0.19711548,-2.6028845)
 			\psline[linecolor=black, linewidth=0.08](6.9971156,-4.2028847)(6.5971155,-4.2028847)(6.5971155,-4.2028847)
 			\psline[linecolor=black, linewidth=0.08](8.5971155,-4.2028847)(8.997115,-4.2028847)(8.997115,-4.2028847)
 			\psdots[linecolor=black, dotsize=0.4](8.997115,-4.2028847)
 			\psdots[linecolor=black, dotsize=0.4](8.997115,-2.6028845)
 			\psdots[linecolor=black, dotsize=0.4](10.5971155,-2.6028845)
 			\psdots[linecolor=black, dotsize=0.4](10.5971155,-4.2028847)
 			\psdots[linecolor=black, dotsize=0.4](10.5971155,-5.8028846)
 			\psdots[linecolor=black, dotsize=0.4](12.197116,-5.8028846)
 			\psdots[linecolor=black, dotsize=0.4](12.197116,-4.2028847)
 			\psline[linecolor=black, linewidth=0.08](8.997115,-4.2028847)(12.197116,-4.2028847)(12.197116,-5.8028846)(10.5971155,-5.8028846)(10.5971155,-2.6028845)(8.997115,-2.6028845)(8.997115,-4.2028847)(8.997115,-4.2028847)
 			\psdots[linecolor=black, dotsize=0.1](7.3971157,-4.2028847)
 			\psdots[linecolor=black, dotsize=0.1](7.7971153,-4.2028847)
 			\psdots[linecolor=black, dotsize=0.1](8.197116,-4.2028847)
 			\end{pspicture}
 		}
 	\end{minipage}
 	\begin{minipage}{7.5cm}
 		\psscalebox{0.45 0.45}
 		{
 			\begin{pspicture}(0,-6.8)(13.194231,-1.605769)
 			\psdots[linecolor=black, dotsize=0.4](2.1971154,-1.8028846)
 			\psdots[linecolor=black, dotsize=0.4](2.1971154,-4.2028847)
 			\psdots[linecolor=black, dotsize=0.4](2.5971155,-3.0028846)
 			\psdots[linecolor=black, dotsize=0.4](3.3971155,-3.0028846)
 			\psdots[linecolor=black, dotsize=0.4](3.7971156,-4.2028847)
 			\psdots[linecolor=black, dotsize=0.4](3.7971156,-1.8028846)
 			\psdots[linecolor=black, dotsize=0.4](5.3971157,-1.8028846)
 			\psdots[linecolor=black, dotsize=0.4](5.7971153,-3.0028846)
 			\psdots[linecolor=black, dotsize=0.4](5.3971157,-4.2028847)
 			\psdots[linecolor=black, dotsize=0.4](0.59711546,-1.8028846)
 			\psdots[linecolor=black, dotsize=0.4](0.19711548,-3.0028846)
 			\psdots[linecolor=black, dotsize=0.4](0.59711546,-4.2028847)
 			\psdots[linecolor=black, dotsize=0.4](1.7971154,-5.4028845)
 			\psdots[linecolor=black, dotsize=0.4](2.1971154,-6.6028843)
 			\psdots[linecolor=black, dotsize=0.4](3.7971156,-6.6028843)
 			\psdots[linecolor=black, dotsize=0.4](4.1971154,-5.4028845)
 			\psdots[linecolor=black, dotsize=0.4](6.9971156,-4.2028847)
 			\psdots[linecolor=black, dotsize=0.4](4.9971156,-5.4028845)
 			\psdots[linecolor=black, dotsize=0.4](5.3971157,-6.6028843)
 			\psdots[linecolor=black, dotsize=0.4](7.3971157,-5.4028845)
 			\psdots[linecolor=black, dotsize=0.4](6.9971156,-6.6028843)
 			\psdots[linecolor=black, dotsize=0.4](10.997115,-1.8028846)
 			\psdots[linecolor=black, dotsize=0.4](10.997115,-4.2028847)
 			\psdots[linecolor=black, dotsize=0.4](11.397116,-3.0028846)
 			\psdots[linecolor=black, dotsize=0.4](12.5971155,-4.2028847)
 			\psdots[linecolor=black, dotsize=0.4](9.397116,-1.8028846)
 			\psdots[linecolor=black, dotsize=0.4](8.997115,-3.0028846)
 			\psdots[linecolor=black, dotsize=0.4](9.397116,-4.2028847)
 			\psdots[linecolor=black, dotsize=0.4](10.5971155,-5.4028845)
 			\psdots[linecolor=black, dotsize=0.4](10.997115,-6.6028843)
 			\psdots[linecolor=black, dotsize=0.4](12.5971155,-6.6028843)
 			\psdots[linecolor=black, dotsize=0.4](12.997115,-5.4028845)
 			\psline[linecolor=black, linewidth=0.08](6.9971156,-4.2028847)(0.59711546,-4.2028847)(0.19711548,-3.0028846)(0.59711546,-1.8028846)(2.1971154,-1.8028846)(2.5971155,-3.0028846)(2.1971154,-4.2028847)(1.7971154,-5.4028845)(2.1971154,-6.6028843)(3.7971156,-6.6028843)(4.1971154,-5.4028845)(3.7971156,-4.2028847)(3.3971155,-3.0028846)(3.7971156,-1.8028846)(5.3971157,-1.8028846)(5.7971153,-3.0028846)(5.3971157,-4.2028847)(4.9971156,-5.4028845)(5.3971157,-6.6028843)(6.9971156,-6.6028843)(7.3971157,-5.4028845)(6.9971156,-4.2028847)(6.9971156,-4.2028847)
 			\psline[linecolor=black, linewidth=0.08](9.397116,-4.2028847)(12.5971155,-4.2028847)(12.997115,-5.4028845)(12.5971155,-6.6028843)(10.997115,-6.6028843)(10.5971155,-5.4028845)(11.397116,-3.0028846)(10.997115,-1.8028846)(9.397116,-1.8028846)(8.997115,-3.0028846)(9.397116,-4.2028847)(9.397116,-4.2028847)
 			\psline[linecolor=black, linewidth=0.08](7.3971157,-4.2028847)(6.9971156,-4.2028847)(6.9971156,-4.2028847)
 			\psline[linecolor=black, linewidth=0.08](8.997115,-4.2028847)(9.397116,-4.2028847)(9.397116,-4.2028847)
 			\psdots[linecolor=black, dotsize=0.1](8.197116,-4.2028847)
 			\psdots[linecolor=black, dotsize=0.1](7.7971153,-4.2028847)
 			\psdots[linecolor=black, dotsize=0.1](8.5971155,-4.2028847)
 			\end{pspicture}
 		}
 	\end{minipage}
 	\caption{Para-chain square cactus $O_n$ and ortho-chain graph $O_n^h$ } \label{ortho-ohn}
 \end{figure}
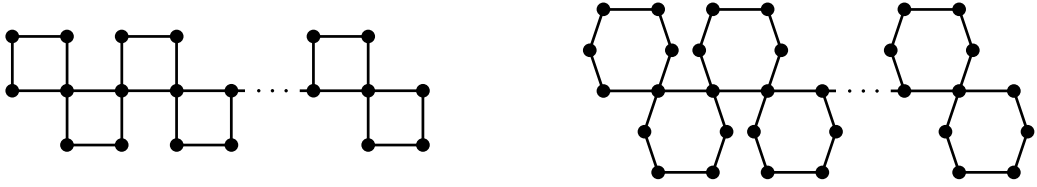

 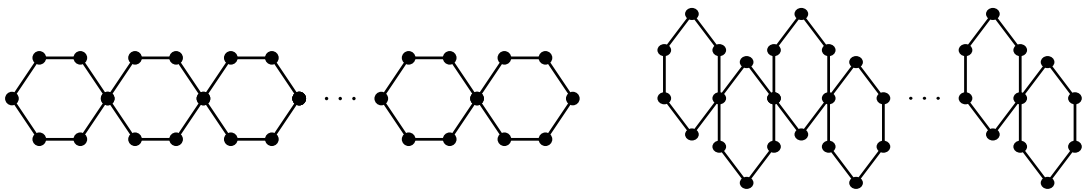
\begin{figure}
 \hspace{0.7cm}	\begin{minipage}{7.5cm}
 		\psscalebox{0.45 0.45}
 		{
 			\begin{pspicture}(0,-5.6)(16.794231,-2.805769)
 			\psdots[linecolor=black, dotsize=0.4](2.1971154,-3.0028846)
 			\psdots[linecolor=black, dotsize=0.4](2.1971154,-5.4028845)
 			\psdots[linecolor=black, dotsize=0.4](2.9971154,-4.2028847)
 			\psdots[linecolor=black, dotsize=0.4](2.9971154,-4.2028847)
 			\psdots[linecolor=black, dotsize=0.4](3.7971153,-5.4028845)
 			\psdots[linecolor=black, dotsize=0.4](3.7971153,-3.0028846)
 			\psdots[linecolor=black, dotsize=0.4](4.9971156,-3.0028846)
 			\psdots[linecolor=black, dotsize=0.4](5.7971153,-4.2028847)
 			\psdots[linecolor=black, dotsize=0.4](4.9971156,-5.4028845)
 			\psdots[linecolor=black, dotsize=0.4](0.9971154,-3.0028846)
 			\psdots[linecolor=black, dotsize=0.4](0.19711538,-4.2028847)
 			\psdots[linecolor=black, dotsize=0.4](0.9971154,-5.4028845)
 			\psdots[linecolor=black, dotsize=0.4](7.7971153,-3.0028846)
 			\psdots[linecolor=black, dotsize=0.4](7.7971153,-5.4028845)
 			\psdots[linecolor=black, dotsize=0.4](8.5971155,-4.2028847)
 			\psdots[linecolor=black, dotsize=0.4](8.5971155,-4.2028847)
 			\psdots[linecolor=black, dotsize=0.4](6.5971155,-3.0028846)
 			\psdots[linecolor=black, dotsize=0.4](5.7971153,-4.2028847)
 			\psdots[linecolor=black, dotsize=0.4](6.5971155,-5.4028845)
 			\psdots[linecolor=black, dotsize=0.4](12.997115,-3.0028846)
 			\psdots[linecolor=black, dotsize=0.4](12.997115,-5.4028845)
 			\psdots[linecolor=black, dotsize=0.4](13.797115,-4.2028847)
 			\psdots[linecolor=black, dotsize=0.4](13.797115,-4.2028847)
 			\psdots[linecolor=black, dotsize=0.4](14.5971155,-5.4028845)
 			\psdots[linecolor=black, dotsize=0.4](14.5971155,-3.0028846)
 			\psdots[linecolor=black, dotsize=0.4](15.797115,-3.0028846)
 			\psdots[linecolor=black, dotsize=0.4](16.597115,-4.2028847)
 			\psdots[linecolor=black, dotsize=0.4](15.797115,-5.4028845)
 			\psdots[linecolor=black, dotsize=0.4](11.797115,-3.0028846)
 			\psdots[linecolor=black, dotsize=0.4](10.997115,-4.2028847)
 			\psdots[linecolor=black, dotsize=0.4](11.797115,-5.4028845)
 			\psdots[linecolor=black, dotsize=0.4](8.5971155,-4.2028847)
 			\psdots[linecolor=black, dotsize=0.4](8.5971155,-4.2028847)
 			\psline[linecolor=black, linewidth=0.08](0.19711538,-4.2028847)(0.9971154,-3.0028846)(2.1971154,-3.0028846)(2.9971154,-4.2028847)(3.7971153,-3.0028846)(4.9971156,-3.0028846)(5.7971153,-4.2028847)(6.5971155,-3.0028846)(7.7971153,-3.0028846)(8.5971155,-4.2028847)(7.7971153,-5.4028845)(6.5971155,-5.4028845)(5.7971153,-4.2028847)(4.9971156,-5.4028845)(3.7971153,-5.4028845)(2.9971154,-4.2028847)(2.1971154,-5.4028845)(0.9971154,-5.4028845)(0.19711538,-4.2028847)(0.19711538,-4.2028847)
 			\psline[linecolor=black, linewidth=0.08](10.997115,-4.2028847)(11.797115,-3.0028846)(12.997115,-3.0028846)(13.797115,-4.2028847)(14.5971155,-3.0028846)(15.797115,-3.0028846)(16.597115,-4.2028847)(15.797115,-5.4028845)(14.5971155,-5.4028845)(13.797115,-4.2028847)(12.997115,-5.4028845)(11.797115,-5.4028845)(10.997115,-4.2028847)(10.997115,-4.2028847)
 			\psdots[linecolor=black, dotsize=0.1](9.397116,-4.2028847)
 			\psdots[linecolor=black, dotsize=0.1](9.797115,-4.2028847)
 			\psdots[linecolor=black, dotsize=0.1](10.197115,-4.2028847)
 			\end{pspicture}
 		}
 	\end{minipage}
 	\hspace{0.9cm}
 	\begin{minipage}{7.5cm} 
 		\psscalebox{0.45 0.40}
 		{
 			\begin{pspicture}(0,-6.4)(12.394231,-0.40576905)
 			\psdots[linecolor=black, dotsize=0.4](0.9971155,-0.60288453)
 			\psdots[linecolor=black, dotsize=0.4](1.7971154,-1.8028846)
 			\psdots[linecolor=black, dotsize=0.4](1.7971154,-3.4028845)
 			\psdots[linecolor=black, dotsize=0.4](0.9971155,-4.6028843)
 			\psdots[linecolor=black, dotsize=0.4](0.19711548,-1.8028846)
 			\psdots[linecolor=black, dotsize=0.4](0.19711548,-3.4028845)
 			\psdots[linecolor=black, dotsize=0.4](2.5971155,-2.2028844)
 			\psdots[linecolor=black, dotsize=0.4](3.3971155,-3.4028845)
 			\psdots[linecolor=black, dotsize=0.4](1.7971154,-5.0028844)
 			\psdots[linecolor=black, dotsize=0.4](3.3971155,-5.0028844)
 			\psdots[linecolor=black, dotsize=0.4](2.5971155,-6.2028847)
 			\psdots[linecolor=black, dotsize=0.4](4.1971154,-0.60288453)
 			\psdots[linecolor=black, dotsize=0.4](4.9971156,-1.8028846)
 			\psdots[linecolor=black, dotsize=0.4](4.9971156,-3.4028845)
 			\psdots[linecolor=black, dotsize=0.4](4.1971154,-4.6028843)
 			\psdots[linecolor=black, dotsize=0.4](3.3971155,-1.8028846)
 			\psdots[linecolor=black, dotsize=0.4](3.3971155,-3.4028845)
 			\psdots[linecolor=black, dotsize=0.4](5.7971153,-2.2028844)
 			\psdots[linecolor=black, dotsize=0.4](6.5971155,-3.4028845)
 			\psdots[linecolor=black, dotsize=0.4](4.9971156,-5.0028844)
 			\psdots[linecolor=black, dotsize=0.4](6.5971155,-5.0028844)
 			\psdots[linecolor=black, dotsize=0.4](5.7971153,-6.2028847)
 			\psdots[linecolor=black, dotsize=0.1](7.3971157,-3.4028845)
 			\psdots[linecolor=black, dotsize=0.1](7.7971153,-3.4028845)
 			\psdots[linecolor=black, dotsize=0.1](8.197116,-3.4028845)
 			\psdots[linecolor=black, dotsize=0.4](9.797115,-0.60288453)
 			\psdots[linecolor=black, dotsize=0.4](10.5971155,-1.8028846)
 			\psdots[linecolor=black, dotsize=0.4](10.5971155,-3.4028845)
 			\psdots[linecolor=black, dotsize=0.4](9.797115,-4.6028843)
 			\psdots[linecolor=black, dotsize=0.4](8.997115,-1.8028846)
 			\psdots[linecolor=black, dotsize=0.4](8.997115,-3.4028845)
 			\psdots[linecolor=black, dotsize=0.4](11.397116,-2.2028844)
 			\psdots[linecolor=black, dotsize=0.4](12.197116,-3.4028845)
 			\psdots[linecolor=black, dotsize=0.4](10.5971155,-5.0028844)
 			\psdots[linecolor=black, dotsize=0.4](12.197116,-5.0028844)
 			\psdots[linecolor=black, dotsize=0.4](11.397116,-6.2028847)
 			\psline[linecolor=black, linewidth=0.08](0.9971155,-0.60288453)(1.7971154,-1.8028846)(1.7971154,-5.0028844)(2.5971155,-6.2028847)(3.3971155,-5.0028844)(3.3971155,-1.8028846)(4.1971154,-0.60288453)(4.9971156,-1.8028846)(4.9971156,-5.0028844)(5.7971153,-6.2028847)(6.5971155,-5.0028844)(6.5971155,-3.4028845)(5.7971153,-2.2028844)(4.9971156,-3.4028845)(4.1971154,-4.6028843)(2.5971155,-2.2028844)(0.9971155,-4.6028843)(0.19711548,-3.4028845)(0.19711548,-1.8028846)(0.9971155,-0.60288453)(0.9971155,-0.60288453)
 			\psline[linecolor=black, linewidth=0.08](11.397116,-2.2028844)(9.797115,-4.6028843)(8.997115,-3.4028845)(8.997115,-1.8028846)(9.797115,-0.60288453)(10.5971155,-1.8028846)(10.5971155,-5.0028844)(11.397116,-6.2028847)(12.197116,-5.0028844)(12.197116,-3.4028845)(11.397116,-2.2028844)(11.397116,-2.2028844)
 			\end{pspicture}
 		}
 	\end{minipage}
 	\caption{Para-chain  $L_n$ and Meta-chain  $M_n$} \label{metaChainMn}
 \end{figure}

\subsection{Polyphenylenes}
Similar to the above definition of the spiro-chain $S_{q,h,k}$, we can define the graph 
$L_{q,h,k}$ 
as the link of $k$ cycles $C_q$ in which the distance between the two contact vertices in the
same cycle is $h$ (see $L_{6,2,4}$ in Figure \ref{L624}). 
	\begin{theorem}
	For the graph  $L_{q,h,k}$, when $h\geq 2$, we have:
\begin{align*}
SO(L_{q,h,k})=(2qk-5k+5)\sqrt{2}+(4k-4)\sqrt{13}.
	\end{align*}
	\end{theorem}
	
		\begin{proof}
	There are $k-1$ edges with endpoints of degree 3. Also there are $4(k-1)$ edges with endpoints of degree 3 and 2, and there are $qk-4(k-1)$ edges with endpoints of degree 2. Therefore 
	\begin{align*}
	SO(L_{q,h,k})=(k-1)\sqrt{9+9}+4(k-1)\sqrt{9+4}+ (qk-4(k-1))\sqrt{4+4},
	\end{align*}
	and we have the result.	
	\qed
	\end{proof}

	\begin{theorem}
	For the graph  $L_{q,1,k}$, we have:
\begin{align*}
SO(L_{q,1,k})=(2qk-5)\sqrt{2}+2k\sqrt{13}.
	\end{align*}
	\end{theorem}

	\begin{proof}
	There are $2k-3$ edges with endpoints of degree 3. Also there are $2k$ edges with endpoints of degree 3 and 2, and there are $qk-3k+2$ edges with endpoints of degree 2. Therefore 
	\begin{align*}
	SO(L_{q,1,k})=(2k-3)\sqrt{9+9}+2k\sqrt{9+4}+ (qk-3k+2)\sqrt{4+4},
	\end{align*}
	and we have the result.	
	\qed
	\end{proof}	

\begin{figure}
\begin{center}
\psscalebox{0.6 0.6}
{
\begin{pspicture}(0,-3.8000002)(10.394231,-0.6057692)
\psdots[linecolor=black, dotsize=0.4](0.9971155,-0.8028845)
\psdots[linecolor=black, dotsize=0.4](0.19711548,-1.6028845)
\psdots[linecolor=black, dotsize=0.4](1.7971154,-1.6028845)
\psdots[linecolor=black, dotsize=0.4](0.19711548,-2.8028846)
\psdots[linecolor=black, dotsize=0.4](1.7971154,-2.8028846)
\psdots[linecolor=black, dotsize=0.4](0.9971155,-3.6028845)
\psdots[linecolor=black, dotsize=0.4](2.9971154,-2.8028846)
\psline[linecolor=black, linewidth=0.08](0.9971155,-0.8028845)(0.19711548,-1.6028845)(0.19711548,-1.6028845)
\psline[linecolor=black, linewidth=0.08](0.9971155,-0.8028845)(1.7971154,-1.6028845)(1.7971154,-1.6028845)(1.7971154,-2.8028846)(1.7971154,-2.8028846)
\psline[linecolor=black, linewidth=0.08](2.9971154,-2.8028846)(1.7971154,-2.8028846)(0.9971155,-3.6028845)(0.19711548,-2.8028846)(0.19711548,-1.6028845)(0.19711548,-1.6028845)
\psdots[linecolor=black, dotsize=0.4](3.7971156,-0.8028845)
\psdots[linecolor=black, dotsize=0.4](2.9971154,-1.6028845)
\psdots[linecolor=black, dotsize=0.4](4.5971155,-1.6028845)
\psdots[linecolor=black, dotsize=0.4](2.9971154,-2.8028846)
\psdots[linecolor=black, dotsize=0.4](4.5971155,-2.8028846)
\psdots[linecolor=black, dotsize=0.4](3.7971156,-3.6028845)
\psdots[linecolor=black, dotsize=0.4](5.7971153,-2.8028846)
\psline[linecolor=black, linewidth=0.08](3.7971156,-0.8028845)(2.9971154,-1.6028845)(2.9971154,-1.6028845)
\psline[linecolor=black, linewidth=0.08](3.7971156,-0.8028845)(4.5971155,-1.6028845)(4.5971155,-1.6028845)(4.5971155,-2.8028846)(4.5971155,-2.8028846)
\psline[linecolor=black, linewidth=0.08](5.7971153,-2.8028846)(4.5971155,-2.8028846)(3.7971156,-3.6028845)(2.9971154,-2.8028846)(2.9971154,-1.6028845)(2.9971154,-1.6028845)
\psdots[linecolor=black, dotsize=0.4](6.5971155,-0.8028845)
\psdots[linecolor=black, dotsize=0.4](5.7971153,-1.6028845)
\psdots[linecolor=black, dotsize=0.4](7.3971157,-1.6028845)
\psdots[linecolor=black, dotsize=0.4](5.7971153,-2.8028846)
\psdots[linecolor=black, dotsize=0.4](7.3971157,-2.8028846)
\psdots[linecolor=black, dotsize=0.4](6.5971155,-3.6028845)
\psdots[linecolor=black, dotsize=0.4](8.5971155,-2.8028846)
\psline[linecolor=black, linewidth=0.08](6.5971155,-0.8028845)(5.7971153,-1.6028845)(5.7971153,-1.6028845)
\psline[linecolor=black, linewidth=0.08](6.5971155,-0.8028845)(7.3971157,-1.6028845)(7.3971157,-1.6028845)(7.3971157,-2.8028846)(7.3971157,-2.8028846)
\psline[linecolor=black, linewidth=0.08](8.5971155,-2.8028846)(7.3971157,-2.8028846)(6.5971155,-3.6028845)(5.7971153,-2.8028846)(5.7971153,-1.6028845)(5.7971153,-1.6028845)
\psdots[linecolor=black, dotsize=0.4](9.397116,-0.8028845)
\psdots[linecolor=black, dotsize=0.4](8.5971155,-1.6028845)
\psdots[linecolor=black, dotsize=0.4](10.197116,-1.6028845)
\psdots[linecolor=black, dotsize=0.4](8.5971155,-2.8028846)
\psdots[linecolor=black, dotsize=0.4](10.197116,-2.8028846)
\psdots[linecolor=black, dotsize=0.4](9.397116,-3.6028845)
\psline[linecolor=black, linewidth=0.08](9.397116,-0.8028845)(8.5971155,-1.6028845)(8.5971155,-1.6028845)
\psline[linecolor=black, linewidth=0.08](9.397116,-0.8028845)(10.197116,-1.6028845)(10.197116,-1.6028845)(10.197116,-2.8028846)(10.197116,-2.8028846)
\psline[linecolor=black, linewidth=0.08](10.197116,-2.8028846)(9.397116,-3.6028845)(8.5971155,-2.8028846)(8.5971155,-1.6028845)(8.5971155,-1.6028845)
\end{pspicture}
}
\end{center}
\caption{The graph $L_{6,2,4}$}\label{L624}
\end{figure}
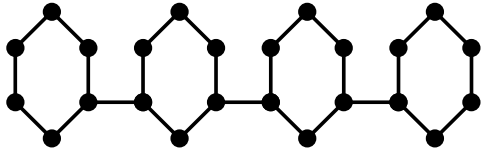

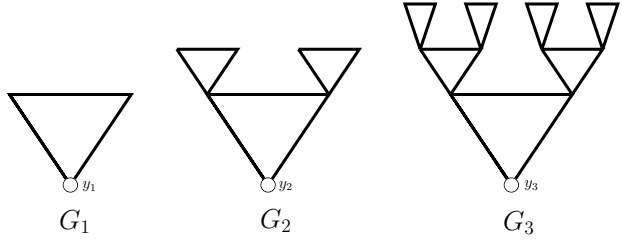
\begin{figure}
	\begin{center}
		\psscalebox{0.5 0.5}
{
\begin{pspicture}(0,-7.8216667)(16.150236,-1.6783332)
\psline[linecolor=black, linewidth=0.08](7.6847405,-2.928333)(9.28474,-2.928333)(8.48474,-4.128333)(7.6847405,-2.928333)(7.6847405,-2.928333)
\psline[linecolor=black, linewidth=0.08](10.884741,-2.928333)(11.28474,-1.7283331)(10.48474,-1.7283331)(10.884741,-2.928333)(10.884741,-2.928333)
\psline[linecolor=black, linewidth=0.08](5.2847404,-4.128333)(6.884741,-6.528333)(8.48474,-4.128333)(5.2847404,-4.128333)(6.884741,-6.528333)(6.884741,-6.528333)
\psline[linecolor=black, linewidth=0.08](0.0847406,-4.128333)(1.6847405,-6.528333)(3.2847407,-4.128333)(0.0847406,-4.128333)(1.6847405,-6.528333)(1.6847405,-6.528333)
\psline[linecolor=black, linewidth=0.08](4.4847407,-2.928333)(6.0847406,-2.928333)(5.2847404,-4.128333)(4.4847407,-2.928333)(4.4847407,-2.928333)
\psline[linecolor=black, linewidth=0.08](14.084741,-2.928333)(15.684741,-2.928333)(14.884741,-4.128333)(14.084741,-2.928333)(14.084741,-2.928333)
\psline[linecolor=black, linewidth=0.08](11.684741,-4.128333)(13.28474,-6.528333)(14.884741,-4.128333)(11.684741,-4.128333)(13.28474,-6.528333)(13.28474,-6.528333)
\psline[linecolor=black, linewidth=0.08](10.884741,-2.928333)(12.48474,-2.928333)(11.684741,-4.128333)(10.884741,-2.928333)(10.884741,-2.928333)
\psline[linecolor=black, linewidth=0.08](12.48474,-2.928333)(12.884741,-1.7283331)(12.084741,-1.7283331)(12.48474,-2.928333)(12.48474,-2.928333)
\psline[linecolor=black, linewidth=0.08](14.084741,-2.928333)(14.48474,-1.7283331)(13.684741,-1.7283331)(14.084741,-2.928333)(14.084741,-2.928333)
\psline[linecolor=black, linewidth=0.08](15.684741,-2.928333)(16.08474,-1.7283331)(15.28474,-1.7283331)(15.684741,-2.928333)(15.684741,-2.928333)
\psdots[linecolor=black, dotstyle=o, dotsize=0.4, fillcolor=white](1.6847405,-6.528333)
\psdots[linecolor=black, dotstyle=o, dotsize=0.4, fillcolor=white](6.884741,-6.528333)
\psdots[linecolor=black, dotstyle=o, dotsize=0.4, fillcolor=white](13.28474,-6.528333)
\rput[bl](2.0047407,-6.7016664){$y_1$}
\rput[bl](7.1647406,-6.688333){$y_2$}
\rput[bl](13.631408,-6.6749997){$y_3$}
\rput[bl](1.1647406,-7.808333){\begin{LARGE}
$G_1$
\end{LARGE}}
\rput[bl](6.458074,-7.768333){\begin{LARGE}
$G_2$
\end{LARGE}}
\rput[bl](12.844741,-7.8216662){\begin{LARGE}
$G_3$
\end{LARGE}}
\end{pspicture}
}
	\end{center}
	\caption{Graphs $G_1$, $G_2$ and $G_3$}\label{triang}
\end{figure}

\subsection{Triangulanes}
We intend to derive the Sombor index of the triangulane $T_k$ defined pictorially in \cite{Khalifeh}.
We define $T_k$ recursively in a manner that will be useful in our approach. First we define
recursively an auxiliary family of triangulanes $G_k$ $(k\geq 1)$. Let $G_1$ be a triangle and denote one of its vertices by $y_1$. We define $G_k$ $(k\geq 2)$ as the circuit of the graphs $G_{k-1}, G_{k-1}$,
and $K_1$ and denote by $y_k$ the vertex where $K_1$ has been placed. The graphs $G_1, G_2$ and $G_3$ 
are shown in Figure \ref{triang}.

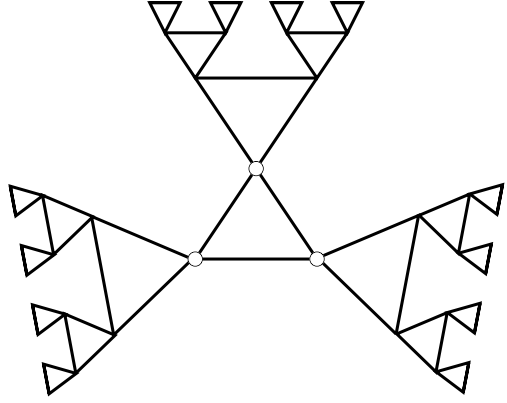
\begin{figure}
	\begin{center}
		\psscalebox{0.5 0.5}
		{
			\begin{pspicture}(0,-8.002529)(13.070843,2.5025294)
			\psline[linecolor=black, linewidth=0.08](3.7284591,2.4525292)(4.528459,2.4525292)(4.128459,1.6525292)(3.7284591,2.4525292)(4.128459,2.4525292)(4.528459,2.4525292)
			\psline[linecolor=black, linewidth=0.08](5.328459,2.4525292)(6.128459,2.4525292)(5.728459,1.6525292)(5.328459,2.4525292)(5.728459,2.4525292)(6.128459,2.4525292)
			\psline[linecolor=black, linewidth=0.08](4.128459,1.6525292)(5.728459,1.6525292)(4.928459,0.4525293)(4.128459,1.6525292)(4.128459,1.6525292)
			\psline[linecolor=black, linewidth=0.08](6.928459,2.4525292)(7.728459,2.4525292)(7.328459,1.6525292)(6.928459,2.4525292)(7.328459,2.4525292)(7.728459,2.4525292)
			\psline[linecolor=black, linewidth=0.08](8.528459,2.4525292)(9.328459,2.4525292)(8.928459,1.6525292)(8.528459,2.4525292)(8.928459,2.4525292)(9.328459,2.4525292)
			\psline[linecolor=black, linewidth=0.08](7.328459,1.6525292)(8.928459,1.6525292)(8.128459,0.4525293)(7.328459,1.6525292)(7.328459,1.6525292)
			\psline[linecolor=black, linewidth=0.08](8.128459,0.4525293)(4.928459,0.4525293)(6.528459,-1.9474707)(8.128459,0.4525293)(8.128459,0.4525293)
			\psline[linecolor=black, linewidth=0.08](13.0095825,-2.376164)(12.861903,-3.162415)(12.149491,-2.62161)(13.0095825,-2.376164)(12.935742,-2.7692895)(12.861903,-3.162415)
			\psline[linecolor=black, linewidth=0.08](12.714223,-3.948666)(12.566544,-4.734917)(11.854133,-4.194112)(12.714223,-3.948666)(12.640384,-4.3417916)(12.566544,-4.734917)
			\psline[linecolor=black, linewidth=0.08](12.144952,-2.6238508)(11.849592,-4.196353)(10.817896,-3.1885824)(12.144952,-2.6238508)(12.144952,-2.6238508)
			\psline[linecolor=black, linewidth=0.08](12.418864,-5.521168)(12.271184,-6.3074193)(11.558773,-5.766614)(12.418864,-5.521168)(12.345024,-5.914294)(12.271184,-6.3074193)
			\psline[linecolor=black, linewidth=0.08](12.123505,-7.0936704)(11.975825,-7.8799214)(11.263413,-7.339116)(12.123505,-7.0936704)(12.0496645,-7.486796)(11.975825,-7.8799214)
			\psline[linecolor=black, linewidth=0.08](11.554234,-5.768855)(11.258874,-7.341357)(10.227177,-6.3335867)(11.554234,-5.768855)(11.554234,-5.768855)
			\psline[linecolor=black, linewidth=0.08](10.210442,-6.336506)(10.80116,-3.1915019)(8.147048,-4.320965)(10.210442,-6.336506)(10.210442,-6.336506)
			\psline[linecolor=black, linewidth=0.08](1.0839291,-7.92159)(0.93781745,-7.135046)(1.7974172,-7.3822064)(1.0839291,-7.92159)(1.0108732,-7.528318)(0.93781745,-7.135046)
			\psline[linecolor=black, linewidth=0.08](0.7917058,-6.348502)(0.6455942,-5.5619583)(1.505194,-5.8091187)(0.7917058,-6.348502)(0.71865,-5.95523)(0.6455942,-5.5619583)
			\psline[linecolor=black, linewidth=0.08](1.7854072,-7.3759704)(1.4931839,-5.802882)(2.8191113,-6.370259)(1.7854072,-7.3759704)(1.7854072,-7.3759704)
			\psline[linecolor=black, linewidth=0.08](0.49948257,-4.7754145)(0.35337096,-3.9888704)(1.2129707,-4.2360306)(0.49948257,-4.7754145)(0.42642677,-4.382142)(0.35337096,-3.9888704)
			\psline[linecolor=black, linewidth=0.08](0.20725933,-3.2023263)(0.061147712,-2.4157825)(0.92074746,-2.6629426)(0.20725933,-3.2023263)(0.13420352,-2.8090544)(0.061147712,-2.4157825)
			\psline[linecolor=black, linewidth=0.08](1.2009606,-4.2297945)(0.9087374,-2.6567066)(2.234665,-3.224083)(1.2009606,-4.2297945)(1.2009606,-4.2297945)
			\psline[linecolor=black, linewidth=0.08](2.197415,-3.215118)(2.7818615,-6.361294)(4.8492703,-4.3498707)(2.197415,-3.215118)(2.197415,-3.215118)
			\psline[linecolor=black, linewidth=0.08](6.528459,-1.9474707)(4.928459,-4.3474708)(8.128459,-4.3474708)(6.528459,-1.9474707)(6.528459,-1.9474707)
			\psdots[linecolor=black, dotstyle=o, dotsize=0.4, fillcolor=white](6.528459,-1.9474707)
			\psdots[linecolor=black, dotstyle=o, dotsize=0.4, fillcolor=white](4.928459,-4.3474708)
			\psdots[linecolor=black, dotstyle=o, dotsize=0.4, fillcolor=white](8.128459,-4.3474708)
			\end{pspicture}
		}
	\end{center}
	\caption{Graphs $T_3$}\label{T3}
\end{figure}

	\begin{theorem} 
	For the graph  $T_k$ (see $T_3$ in Figure \ref{T3}), we have:
\begin{align*}
SO(T_k)=\big(36(2^{k-1}-1)+6(2^{k-1})+12\big)\sqrt{2}+6(2^k)\sqrt{5}.
	\end{align*}
	\end{theorem}

	\begin{proof}
	Since creating such a graph is recursive, then there are $3+3\sum_{n=0}^{k-2}3(2^n)$ edges with endpoints of degree 4. Also there are $3(2^k)$ edges with endpoints of degree 4 and 2, and there are $3(2^{k-1})$ edges with endpoints of degree 2. Therefore 
	\begin{align*}
	SO(T_k)=(3+9\sum_{n=0}^{k-2}2^n)\sqrt{16+16}+3(2^k)\sqrt{16+4}+ 3(2^{k-1})\sqrt{4+4},
	\end{align*}
	and we have the result.	
	\qed
	\end{proof}

\subsection{Nanostar dendrimers}
We want to compute  the Sombor index of the nanostar dendrimer $D_k$ defined in \cite{MATCH}.  In order to define $D_k$, we follow \cite{Deutsch}. First we define recursively an auxiliary family of rooted dendrimers $G_k$ 
$(k\geq 1)$. We need
a fixed graph $F$ defined in Figure \ref{FG1}, we consider one of its endpoint to be the root of $F$.
The graph $G_1$ is defined in Figure \ref{FG1}, the leaf being its root. Now we define $G_k$ $(k\geq2)$ the bouquet of the following 3 graphs: $G_{k-1}, G_{k-1}$, and $F$ with respect to their roots; the
root of $G_k$ is taken to be its unique leaf (see $G_2$ and $G_3$ in Figure \ref{G2G3}). Finally, we define $D_k$ $(k\geq 1)$ as the bouquet of 3 copies of $G_k$ with respect to their roots ($D_2$ is shown in Figure \ref{nanostar}, where the circles represent hexagons).

\begin{figure}
\hspace{2.7cm}	\begin{minipage}{6.5cm}
		\psscalebox{0.47 0.47}
		{
			\begin{pspicture}(0,-4.8)(5.62,-3.14)
			\psline[linecolor=black, linewidth=0.04](0.01,-3.97)(0.81,-3.97)(1.21,-3.17)(2.01,-3.17)(2.41,-3.97)(2.01,-4.77)(1.21,-4.77)(0.81,-3.97)(0.81,-3.97)
			\psline[linecolor=black, linewidth=0.04](2.41,-3.97)(3.21,-3.97)(3.61,-3.17)(4.41,-3.17)(4.81,-3.97)(4.41,-4.77)(3.61,-4.77)(3.21,-3.97)(3.21,-3.97)
			\psline[linecolor=black, linewidth=0.04](5.21,-3.97)(4.81,-3.97)(4.81,-3.97)
			\psline[linecolor=black, linewidth=0.04](5.61,-3.97)(5.21,-3.97)(5.21,-3.97)
			\end{pspicture}
		}
	\end{minipage}
	\begin{minipage}{5cm}
		\psscalebox{0.47 0.47}
		{
			\begin{pspicture}(0,-4.8)(2.4423609,-3.14)
			\psline[linecolor=black, linewidth=0.04](0.01,-3.97)(0.81,-3.97)(1.21,-3.17)(2.01,-3.17)(2.41,-3.97)(2.01,-4.77)(1.21,-4.77)(0.81,-3.97)(0.81,-3.97)
			\end{pspicture}
		}
	\end{minipage}
	\caption{Graphs $F$ and $G_1$, respectively.}\label{FG1}
\end{figure}

\begin{figure}
\hspace{1.1cm} 	\begin{minipage}{7cm}
		\psscalebox{0.46 0.46}
		{
			\begin{pspicture}(0,-6.79)(4.86,2.85)
			\psline[linecolor=black, linewidth=0.04](2.43,-0.38)(1.63,-1.18)(1.63,-1.98)(2.43,-2.78)(3.23,-1.98)(3.23,-1.18)(2.43,-0.38)(2.43,-0.38)
			\psline[linecolor=black, linewidth=0.04](2.43,-3.58)(2.43,-2.78)(2.43,-2.78)
			\psline[linecolor=black, linewidth=0.04](2.43,-3.58)(1.63,-4.38)(1.63,-5.18)(2.43,-5.98)(3.23,-5.18)(3.23,-4.38)(2.43,-3.58)(2.43,-3.58)
			\psline[linecolor=black, linewidth=0.04](2.43,-6.78)(2.43,-5.98)(2.43,-5.98)
			\psline[linecolor=black, linewidth=0.04](2.43,0.42)(2.43,-0.38)(2.43,-0.38)
			\psline[linecolor=black, linewidth=0.04](2.43,0.42)(3.23,1.22)(3.23,1.22)
			\psline[linecolor=black, linewidth=0.04](2.43,0.42)(1.63,1.22)(1.63,1.22)
			\psline[linecolor=black, linewidth=0.04](3.23,1.22)(3.23,2.02)(4.03,2.82)(4.83,2.82)(4.83,2.02)(4.03,1.22)(3.23,1.22)(3.23,1.22)
			\psline[linecolor=black, linewidth=0.04](1.63,1.22)(1.63,2.02)(0.83,2.82)(0.03,2.82)(0.03,2.02)(0.83,1.22)(1.63,1.22)(1.63,1.22)
			\end{pspicture}
		}
	\end{minipage}
	\begin{minipage}{7cm}
		\psscalebox{0.4 0.4}
		{
			\begin{pspicture}(0,-11.59)(17.276567,4.45)
			\psline[linecolor=black, linewidth=0.04](8.438284,-5.18)(7.638284,-5.98)(7.638284,-6.78)(8.438284,-7.58)(9.238284,-6.78)(9.238284,-5.98)(8.438284,-5.18)(8.438284,-5.18)
			\psline[linecolor=black, linewidth=0.04](8.438284,-8.38)(8.438284,-7.58)(8.438284,-7.58)
			\psline[linecolor=black, linewidth=0.04](8.438284,-8.38)(7.638284,-9.18)(7.638284,-9.98)(8.438284,-10.78)(9.238284,-9.98)(9.238284,-9.18)(8.438284,-8.38)(8.438284,-8.38)
			\psline[linecolor=black, linewidth=0.04](8.438284,-11.58)(8.438284,-10.78)(8.438284,-10.78)
			\psline[linecolor=black, linewidth=0.04](8.438284,-4.38)(8.438284,-5.18)(8.438284,-5.18)
			\psline[linecolor=black, linewidth=0.04](8.438284,-4.38)(9.238284,-3.58)(9.238284,-3.58)
			\psline[linecolor=black, linewidth=0.04](8.438284,-4.38)(7.638284,-3.58)(7.638284,-3.58)
			\psline[linecolor=black, linewidth=0.04](9.238284,-3.58)(9.238284,-2.78)(10.038284,-1.98)(10.838284,-1.98)(10.838284,-2.78)(10.038284,-3.58)(9.238284,-3.58)(9.238284,-3.58)
			\psline[linecolor=black, linewidth=0.04](7.638284,-3.58)(7.638284,-2.78)(6.838284,-1.98)(6.038284,-1.98)(6.038284,-2.78)(6.838284,-3.58)(7.638284,-3.58)(7.638284,-3.58)
			\psline[linecolor=black, linewidth=0.04](10.838284,-1.98)(11.638284,-1.18)(11.638284,-1.18)
			\psline[linecolor=black, linewidth=0.04](11.638284,-1.18)(11.638284,-0.38)(12.438284,0.42)(13.238284,0.42)(13.238284,-0.38)(12.438284,-1.18)(11.638284,-1.18)(11.638284,-1.18)
			\psline[linecolor=black, linewidth=0.04](6.038284,-1.98)(5.238284,-1.18)(5.238284,-1.18)
			\psline[linecolor=black, linewidth=0.04](5.238284,-1.18)(5.238284,-0.38)(4.438284,0.42)(3.638284,0.42)(3.638284,-0.38)(4.438284,-1.18)(5.238284,-1.18)(5.238284,-1.18)
			\psline[linecolor=black, linewidth=0.04](14.038284,4.42)(13.238284,3.62)(13.238284,2.82)(14.038284,2.02)(14.838284,2.82)(14.838284,3.62)(14.038284,4.42)(14.038284,4.42)
			\psline[linecolor=black, linewidth=0.04](14.038284,1.22)(14.038284,2.02)(14.038284,2.02)
			\psline[linecolor=black, linewidth=0.04](14.038284,1.22)(13.238284,0.42)(13.238284,0.42)
			\psline[linecolor=black, linewidth=0.04](14.038284,1.22)(14.838284,1.22)(14.838284,1.22)
			\psline[linecolor=black, linewidth=0.04](3.238284,4.42)(2.438284,3.62)(2.438284,2.82)(3.238284,2.02)(4.038284,2.82)(4.038284,3.62)(3.238284,4.42)(3.238284,4.42)
			\psline[linecolor=black, linewidth=0.04](3.238284,1.22)(3.238284,2.02)(3.238284,2.02)
			\psline[linecolor=black, linewidth=0.04](3.638284,0.42)(3.238284,1.22)(2.438284,1.22)(2.438284,1.22)
			\psline[linecolor=black, linewidth=0.04](14.838284,1.22)(15.638284,2.02)(16.438284,2.02)(17.238283,1.22)(16.438284,0.42)(15.638284,0.42)(14.838284,1.22)(14.838284,1.22)
			\psline[linecolor=black, linewidth=0.04](2.438284,1.22)(1.638284,2.02)(0.838284,2.02)(0.038283996,1.22)(0.838284,0.42)(1.638284,0.42)(2.438284,1.22)(2.438284,1.22)
			\end{pspicture}
		}
	\end{minipage}
	\caption{Graphs $G_2$ and $G_3$, respectively.}\label{G2G3}
\end{figure}
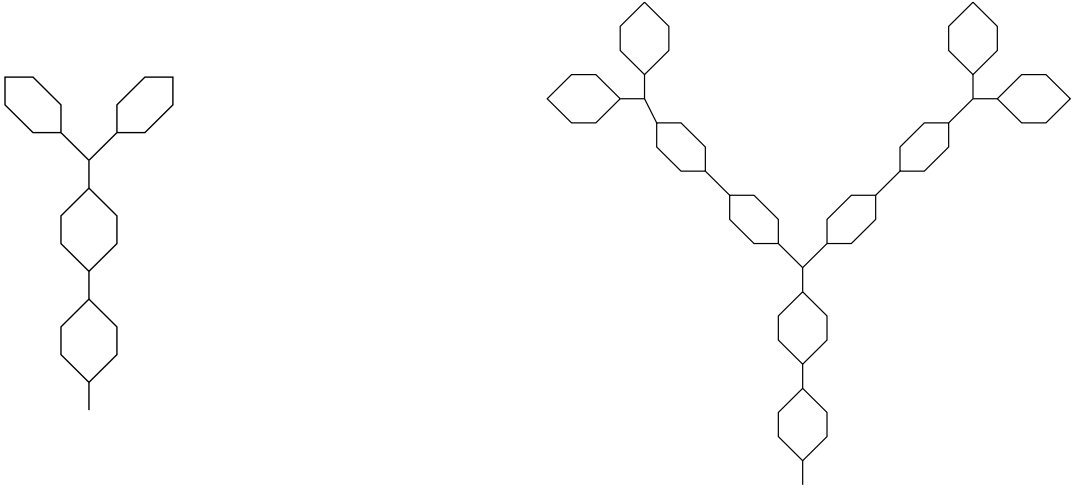

\begin{theorem}
	For the dendrimer $D_3[n]$ (see $D_3[2]$ in Figure \ref{nanostar}) we have:
	\begin{align*}
	SO(D_3[n])=(63\times 2^n-30)\sqrt{2}+(18\times 2^n-12)\sqrt{13}.
	\end{align*}
\end{theorem}

\begin{proof}
	There are $3+9\displaystyle\sum_{k=0}^{n-1}2^k$ edges with endpoints of degree 3. Also there are $6+18\displaystyle\sum_{k=0}^{n-1}2^k$ edges with endpoints of degree 3 and 2, and there are $12+18\displaystyle\sum_{k=0}^{n-1}2^k$ edges with endpoints of degree 2. Therefore 
	\begin{align*}
	SO(D_3[n])=(3+9\displaystyle\sum_{k=0}^{n-1}2^k)\sqrt{9+9}+(6+18\displaystyle\sum_{k=0}^{n-1}2^k)\sqrt{9+4}+ (12+18\displaystyle\sum_{k=0}^{n-1}2^k)\sqrt{4+4},
	\end{align*}
	and we have the result.	
	\qed
\end{proof}

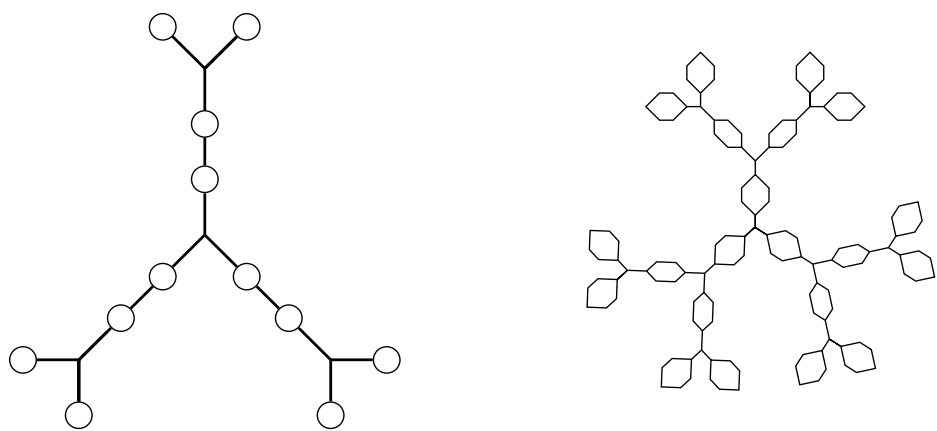
\begin{figure}
\hspace{1cm}	\begin{minipage}{7.5cm}
		\psscalebox{0.46 0.46}
{
\begin{pspicture}(0,-6.0000005)(11.205557,6.005556)
\psdots[linecolor=black, fillstyle=solid, dotstyle=o, dotsize=0.8, fillcolor=white](5.6027775,1.2027783)
\psdots[linecolor=black, fillstyle=solid, dotstyle=o, dotsize=0.8, fillcolor=white](5.6027775,2.8027782)
\psdots[linecolor=black, fillstyle=solid, dotstyle=o, dotsize=0.8, fillcolor=white](10.802777,-3.9972217)
\psdots[linecolor=black, fillstyle=solid, dotstyle=o, dotsize=0.8, fillcolor=white](9.202778,-5.597222)
\psdots[linecolor=black, fillstyle=solid, dotstyle=o, dotsize=0.8, fillcolor=white](0.4027777,-3.9972217)
\psdots[linecolor=black, fillstyle=solid, dotstyle=o, dotsize=0.8, fillcolor=white](2.0027778,-5.597222)
\psline[linecolor=black, linewidth=0.08](0.8027777,-3.9972217)(2.0027778,-3.9972217)(2.0027778,-5.1972218)(2.0027778,-3.9972217)(2.8027778,-3.1972218)(2.8027778,-3.1972218)
\psline[linecolor=black, linewidth=0.08](9.202778,-3.9972217)(9.202778,-5.1972218)(9.202778,-5.1972218)
\psline[linecolor=black, linewidth=0.08](9.202778,-3.9972217)(10.402778,-3.9972217)(10.402778,-3.9972217)
\psline[linecolor=black, linewidth=0.08](5.6027775,-0.39722168)(5.6027775,0.8027783)(5.6027775,0.8027783)
\psline[linecolor=black, linewidth=0.08](5.6027775,1.6027783)(5.6027775,2.4027784)(5.6027775,2.4027784)
\psline[linecolor=black, linewidth=0.08](5.6027775,4.402778)(5.6027775,3.2027783)(5.6027775,3.2027783)
\psline[linecolor=black, linewidth=0.08](5.6027775,4.402778)(6.802778,5.6027784)(5.6027775,4.402778)(4.4027777,5.6027784)(4.4027777,5.6027784)
\psdots[linecolor=black, fillstyle=solid, dotstyle=o, dotsize=0.8, fillcolor=white](6.802778,5.6027784)
\psdots[linecolor=black, fillstyle=solid, dotstyle=o, dotsize=0.8, fillcolor=white](4.4027777,5.6027784)
\psline[linecolor=black, linewidth=0.08](2.0027778,-3.9972217)(5.6027775,-0.39722168)(5.6027775,-0.39722168)(9.202778,-3.9972217)(9.202778,-3.9972217)
\psdots[linecolor=black, fillstyle=solid, dotstyle=o, dotsize=0.8, fillcolor=white](6.802778,-1.5972217)
\psdots[linecolor=black, fillstyle=solid, dotstyle=o, dotsize=0.8, fillcolor=white](8.002778,-2.7972217)
\psdots[linecolor=black, fillstyle=solid, dotstyle=o, dotsize=0.8, fillcolor=white](4.4027777,-1.5972217)
\psdots[linecolor=black, fillstyle=solid, dotstyle=o, dotsize=0.8, fillcolor=white](3.2027776,-2.7972217)
\end{pspicture}
}
	\end{minipage}
		\begin{minipage}{7.5cm}
		\psscalebox{0.45 0.45}
{
\begin{pspicture}(0,-7.3653517)(10.220022,2.6543653)
\psline[linecolor=black, linewidth=0.04](4.946707,-2.1839187)(4.946707,-2.5839188)(4.946707,-2.1839187)
\psline[linecolor=black, linewidth=0.04](4.946707,-2.1839187)(4.5467067,-1.7839187)(4.5467067,-1.3839188)(4.946707,-0.9839188)(5.3467064,-1.3839188)(5.3467064,-1.7839187)(4.946707,-2.1839187)(4.946707,-2.1839187)
\psline[linecolor=black, linewidth=0.04](4.946707,-0.9839188)(4.946707,-0.58391875)(4.946707,-0.58391875)
\psline[linecolor=black, linewidth=0.04](4.946707,-0.58391875)(5.3467064,-0.18391877)(5.3467064,-0.18391877)
\psline[linecolor=black, linewidth=0.04](4.946707,-0.58391875)(4.5467067,-0.18391877)(4.5467067,-0.18391877)
\psline[linecolor=black, linewidth=0.04](5.3467064,-0.18391877)(5.3467064,0.21608122)(5.7467065,0.61608124)(6.1467066,0.61608124)(6.1467066,0.21608122)(5.7467065,-0.18391877)(5.3467064,-0.18391877)(5.3467064,-0.18391877)
\psline[linecolor=black, linewidth=0.04](4.5467067,-0.18391877)(4.5467067,0.21608122)(4.1467066,0.61608124)(3.7467065,0.61608124)(3.7467065,0.21608122)(4.1467066,-0.18391877)(4.5467067,-0.18391877)(4.5467067,-0.18391877)
\psline[linecolor=black, linewidth=0.04](6.1467066,0.61608124)(6.5467067,1.0160812)(6.5467067,1.0160812)
\psline[linecolor=black, linewidth=0.04](3.7467065,0.61608124)(3.3467066,1.0160812)(3.3467066,1.0160812)
\psline[linecolor=black, linewidth=0.04](6.5479274,1.016033)(6.5479274,1.4160331)(6.5479274,1.016033)(6.9479275,1.016033)(6.9479275,1.016033)
\psline[linecolor=black, linewidth=0.04](6.5467067,1.4160812)(6.1467066,1.8160812)(6.1467066,2.2160811)(6.5467067,2.6160812)(6.946707,2.2160811)(6.946707,1.8160812)(6.5467067,1.4160812)(6.5467067,1.4160812)
\psline[linecolor=black, linewidth=0.04](6.946707,1.0160812)(7.3467064,1.4160812)(7.7467065,1.4160812)(8.146707,1.0160812)(7.7467065,0.61608124)(7.3467064,0.61608124)(6.946707,1.0160812)(6.946707,1.0160812)
\psline[linecolor=black, linewidth=0.04](3.346579,1.0147685)(3.346579,1.4147685)(3.7465792,1.8147686)(3.7465792,2.2147684)(3.346579,2.6147685)(2.946579,2.2147684)(2.946579,1.8147686)(3.346579,1.4147685)(3.346579,1.4147685)
\psline[linecolor=black, linewidth=0.04](3.3480194,1.0159538)(2.9480193,1.0159538)(2.5480192,1.4159538)(2.1480193,1.4159538)(1.7480192,1.0159538)(2.1480193,0.61595374)(2.5480192,0.61595374)(2.9480193,1.0159538)(2.9480193,1.0159538)
\psline[linecolor=black, linewidth=0.04](5.281036,-2.760522)(4.9442973,-2.5446353)(5.281036,-2.760522)
\psline[linecolor=black, linewidth=0.04](5.281036,-2.760522)(5.8336616,-2.6396697)(6.1704,-2.8555562)(6.2912526,-3.4081817)(5.738627,-3.529034)(5.4018884,-3.3131473)(5.281036,-2.760522)(5.281036,-2.760522)
\psline[linecolor=black, linewidth=0.04](6.2912526,-3.4081817)(6.627991,-3.624068)(6.627991,-3.624068)
\psline[linecolor=black, linewidth=0.04](6.627991,-3.624068)(6.748843,-4.1766934)(6.748843,-4.1766934)
\psline[linecolor=black, linewidth=0.04](6.627991,-3.624068)(7.1806164,-3.5032158)(7.1806164,-3.5032158)
\psline[linecolor=black, linewidth=0.04](6.748843,-4.1766934)(7.0855823,-4.39258)(7.2064342,-4.945205)(6.9905477,-5.2819443)(6.653809,-5.0660577)(6.532957,-4.513432)(6.748843,-4.1766934)(6.748843,-4.1766934)
\psline[linecolor=black, linewidth=0.04](7.1806164,-3.5032158)(7.5173554,-3.7191024)(8.069981,-3.5982502)(8.285867,-3.2615113)(7.949128,-3.045625)(7.396503,-3.1664772)(7.1806164,-3.5032158)(7.1806164,-3.5032158)
\psline[linecolor=black, linewidth=0.04](6.9905477,-5.2819443)(7.1114,-5.8345695)(7.1114,-5.8345695)
\psline[linecolor=black, linewidth=0.04](8.285867,-3.2615113)(8.838492,-3.140659)(8.838492,-3.140659)
\psline[linecolor=black, linewidth=0.04](7.1118913,-5.8347993)(7.44863,-6.050686)(7.1118913,-5.8347993)(6.8960047,-6.171538)(6.8960047,-6.171538)
\psline[linecolor=black, linewidth=0.04](7.4481387,-6.050456)(8.000764,-5.9296036)(8.3375025,-6.14549)(8.458355,-6.6981153)(7.90573,-6.818968)(7.568991,-6.603081)(7.4481387,-6.050456)(7.4481387,-6.050456)
\psline[linecolor=black, linewidth=0.04](6.8955135,-6.171308)(7.0163655,-6.7239337)(6.8004794,-7.0606723)(6.2478538,-7.1815243)(6.127002,-6.628899)(6.3428884,-6.2921605)(6.8955135,-6.171308)(6.8955135,-6.171308)
\psline[linecolor=black, linewidth=0.04](8.839256,-3.1404498)(9.175995,-3.3563364)(9.296847,-3.9089615)(9.633586,-4.1248484)(10.186212,-4.003996)(10.065359,-3.4513707)(9.728621,-3.2354841)(9.175995,-3.3563364)(9.175995,-3.3563364)
\psline[linecolor=black, linewidth=0.04](8.838283,-3.1398954)(9.05417,-2.8031566)(9.606794,-2.6823044)(9.822681,-2.3455656)(9.701829,-1.7929403)(9.149204,-1.9137925)(8.933317,-2.2505312)(9.05417,-2.8031566)(9.05417,-2.8031566)
\psline[linecolor=black, linewidth=0.04](4.6404505,-2.8087912)(4.9341197,-2.5372064)(4.6404505,-2.8087912)
\psline[linecolor=black, linewidth=0.04](4.6411376,-2.8091736)(4.619053,-3.3744278)(4.325383,-3.6460125)(3.7601292,-3.6239278)(3.782214,-3.0586736)(4.0758834,-2.7870889)(4.6411376,-2.8091736)(4.6411376,-2.8091736)
\psline[linecolor=black, linewidth=0.04](3.7601292,-3.6239278)(3.4664598,-3.8955126)(3.4664598,-3.8955126)
\psline[linecolor=black, linewidth=0.04](3.4664598,-3.8955126)(2.9012055,-3.8734279)(2.9012055,-3.8734279)
\psline[linecolor=black, linewidth=0.04](3.4664598,-3.8955126)(3.444375,-4.460767)(3.444375,-4.460767)
\psline[linecolor=black, linewidth=0.04](2.9012055,-3.8734279)(2.607536,-4.1450124)(2.0422819,-4.1229277)(1.7706972,-3.8292582)(2.0643668,-3.5576737)(2.6296208,-3.5797584)(2.9012055,-3.8734279)(2.9012055,-3.8734279)
\psline[linecolor=black, linewidth=0.04](3.444375,-4.460767)(3.1507056,-4.7323513)(3.1286209,-5.2976055)(3.4002056,-5.591275)(3.693875,-5.31969)(3.7159598,-4.754436)(3.444375,-4.460767)(3.444375,-4.460767)
\psline[linecolor=black, linewidth=0.04](1.7706972,-3.8292582)(1.2054431,-3.8071735)(1.2054431,-3.8071735)
\psline[linecolor=black, linewidth=0.04](3.4002056,-5.591275)(3.378121,-6.156529)(3.378121,-6.156529)
\psline[linecolor=black, linewidth=0.04](1.2070984,-3.8073893)(0.91342884,-4.078974)(1.2070984,-3.8073893)(0.9355136,-3.5137198)(0.9355136,-3.5137198)
\psline[linecolor=black, linewidth=0.04](0.9117737,-4.0787582)(0.8896889,-4.6440125)(0.59601945,-4.915597)(0.030765306,-4.8935122)(0.052850045,-4.328258)(0.3465195,-4.0566735)(0.9117737,-4.0787582)(0.9117737,-4.0787582)
\psline[linecolor=black, linewidth=0.04](0.9338584,-3.513504)(0.36860424,-3.4914193)(0.09701952,-3.1977499)(0.11910426,-2.6324956)(0.6843584,-2.6545806)(0.9559431,-2.94825)(0.9338584,-3.513504)(0.9338584,-3.513504)
\psline[linecolor=black, linewidth=0.04](3.3747423,-6.157989)(3.0810728,-6.429574)(2.5158186,-6.4074893)(2.2221491,-6.679074)(2.2000644,-7.244328)(2.7653186,-7.2664127)(3.058988,-6.994828)(3.0810728,-6.429574)(3.0810728,-6.429574)
\psline[linecolor=black, linewidth=0.04](3.379581,-6.159908)(3.6511655,-6.453577)(3.6290808,-7.0188313)(3.9006655,-7.3125005)(4.4659195,-7.334585)(4.4880047,-6.7693315)(4.2164197,-6.4756618)(3.6511655,-6.453577)(3.6511655,-6.453577)
\end{pspicture}
}
\end{minipage} 
	\caption{Nanostar $D_2$ and $D_3[2]$, respectively.}\label{nanostar} 
\end{figure}

	\section{Acknowledgements} 
	The authors would like to express their gratitude to the referee for her/his careful reading and
	helpful comments. 
	The second  author would like to thank the Research Council of Norway (NFR Toppforsk Project Number 274526, Parameterized Complexity for Practical Computing) and Department of Informatics, University of
	Bergen for their support. Also he is thankful to Michael Fellows and
	Michal Walicki for conversations.

\end{document}